\documentclass[a4paper,twoside,11pt]{article}

\usepackage{color}
\usepackage{appendix}

\usepackage[latin1]{inputenc}
\usepackage{lmodern}
\usepackage[T1]{fontenc}

\usepackage[a4paper,top=3cm,bottom=3cm,left=3cm,right=3cm,
bindingoffset=5mm]{geometry}

\usepackage{amsmath,amssymb,amsthm}
\usepackage{graphicx}

\usepackage{epstopdf}

\def\Pp{{\mathbb P}}

\newcommand{\uex}{u_\text{ex}}

\newtheorem{thm}{Theorem}[section]
\newtheorem{prop}[thm]{Proposition}
\newtheorem{lem}[thm]{Lemma}
\newtheorem{cor}[thm]{Corollary}

\newtheorem{remark}{Remark}

\renewcommand{\O}{\Omega}

\newcommand{\up}{{u}_{\pi}}

\newcommand\E{{E}}  % element

\newcommand\Th{{\mathcal T}_h}

\newcommand\R{\mathbb{R}}

\newcommand{\VE}{V_{h|\E}}

\def\wP{\Pi}

\def\decapita#1{}
\def\grecomultibold#1#2{\grecobolddef#1\def\secondobold{#2}%
    \ifx#2\finemultibold\let\next\relax\let\secondobold\relax
    \else\let\next\grecomultibold
    \fi\expandafter\next\secondobold}

\def\grecobolddef#1{%
  \edef\dadef{bf\expandafter\decapita\string#1}%
  \expandafter\def\csname\dadef\endcsname{{\neretto #1}}}

\def\neretto#1{\setbox0=\hbox{\mathsurround=0pt$#1$}%
  \kern.02em\copy0 \kern-\wd0
  \kern-.02em\copy0 \kern-\wd0
  \raise.03em\box0 \kern.02em}
\grecomultibold\gamma\xi\sigma\tau\eta\theta\varepsilon\rho\finemultibold
\def\wbox#1;#2;{\vbox{\hrule\hbox{\vrule height#1mm\kern#2mm\vrule
  height#1mm}\hrule}}

\let\phi\varphi

\let\b\b

\newcommand{\ffxi}  {\mbox{\boldmath $\xi$\unboldmath}}

\def\hpoint#1.#2.#3{{\underline{#1}}_{#2}\cdot
 {\underline{\mathop{\smash{#3}\vphantom{{#1}_{#2}}}}}}

\def\npointp#1.#2{{\underline{#1}}\cdot
 {\underline{\mathop{\smash{#2}\vphantom{{#1}}}}}}

\def\beq{\begin{equation}}
\def\enq{\end{equation}}

\def\PR{{\cal R}}
\def\tri{| \! | \! |}

\author{
L. Beir\~ao da Veiga\thanks{
Dipartimento di Matematica e Applicazioni, Universit\`a di Milano--Bicocca,
Via Cozzi 53, I-20153, Milano, Italy,
and IMATI del CNR, Via Ferrata 1, 27100 Pavia, Italy,
lourenco.beirao@unimib.it},
C. Lovadina\thanks{
Dipartimento di Matematica, Universit\`a di Pavia,
and IMATI del CNR, Via Ferrata 1, 27100 Pavia, Italy,
carlo.lovadina@unipv.it},
A. Russo\thanks{
Dipartimento di Matematica e Applicazioni, Universit\`a di Milano--Bicocca,
Via Cozzi 53, I-20153, Milano, Italy,
and IMATI del CNR, Via Ferrata 1, 27100 Pavia, Italy,
alessandro.russo@unimib.it}
}

\date{}

\title{Stability Analysis for the Virtual Element Method}

\begin{document}
% ------------------------------------------------------------------------------------------------------------

\maketitle

\begin{abstract}
We analyse the Virtual Element Methods (VEM) on a simple elliptic model problem, allowing for more general meshes than the one typically considered in the VEM literature. For instance, meshes with arbitrarily small edges (with respect to the parent element diameter), can be dealt with.
Our general approach applies to different choices of the stability form, including, for example, the ``classical'' one introduced in \cite{volley}, and a recent one presented in \cite{wriggers}.
Finally, we show that the stabilization term can be simplified by dropping the contribution of the internal-to-the-element degrees of freedom. The resulting stabilization form, involving only the boundary degrees of freedom, can be used in the VEM scheme without affecting the stability and convergence properties.
The numerical tests are in accordance with the theoretical predictions.
\end{abstract}

{
% ----------------------------------------------------
\section{Introduction}
% ----------------------------------------------------

% \begin{itemize}

%\item  MINOR: check meglio il regularity result in \eqref{eq:app:reg}. Ma dovrebbe valere, vedi ad es. Grisvard, Corollary 4.4.3.8.
%
%% \item Carlo, please dai un check careful (for a double check) a tutte le costanti del Lemma 4.1 e del Teorema 4.2 ... anzi meglio se dai anche tu un final careful reading a tutto!

% \end{itemize}

The virtual element method (VEM) has been introduced recently in \cite{volley,VEM-elasticity,Brezzi:Marini:plates,projectors} as a generalization of the finite element method that allows to make use of general polygonal/polyhedral meshes. The virtual element method, that enjoyed an increasing interest in the recent literature, has been developed in many aspects and applied to many different problems; we here cite only a few works \cite{BeiraoLovaMora,hitchhikers,variable-primal,Benedetto-VEM-2,Berrone-VEM,BFM,Paulino-VEM,Gatica,VemSteklov,Helmo-PPR,wriggers,Vacca-1} in addition to the ones above, without pretending to be exhaustive. We also note that VEM is not the only recent method that can make use of polytopal meshes: we refer, again as a minimal sample list of papers, to \cite{Cangiani:Georgoulis:Houston2014,DiPietro-Ern-1,Droniou-gradient,VEM-topopt,Gillette-1,ST04,POLY37}.

A VEM scheme may be seen as a Galerkin method built by means of two parts:
\begin{enumerate}
 \item a first term strongly consistent on polynomials, which guarantees the accuracy;
 \item a stabilization term, involving a suitably designed bilinear form, typically written as the sum of two contributions:
 $s_E(\cdot,\cdot)= s_E^\circ(\cdot,\cdot) + s_E^\partial(\cdot,\cdot)$. The form $s_E^\circ(\cdot,\cdot)$ uses the interior degrees of freedom, while the form $s_E^\partial(\cdot,\cdot)$ uses the boundary degrees of freedom.
\end{enumerate}

We remark that under the usual assumptions on the polygonal mesh (namely, shape regularity and the property that
the length of each edge is uniformly comparable to the diameter of the parent element), devising and proving the stability features of the form $s_E(\cdot,\cdot)$ is quite simple.
This is the reason why, in the VEM literature, the focus is on describing explicit expression for
$s_E(\cdot,\cdot)$, while the proof of the corresponding stability result is often omitted.
Instead, the stability analysis is more involved if one allows for
more general mesh assumptions (for instance dropping the edge length condition mentioned above).

The present paper focuses on the stability properties of the bilinear form $s_E(\cdot,\cdot)$.
Although the approach we follow is quite general, we here consider the problem and notation of \cite{volley,hitchhikers} in order to keep the presentation clearer.
Our main results are the following.
\begin{itemize}
 \item {\em The development of a new strategy to prove the convergence of the VEM schemes}, which requires weaker stability conditions on $s_E(\cdot,\cdot)$ than the usual ones. Our approach is used to analyse the situations described below.
 \begin{enumerate}
  \item VEM schemes using a sequence of meshes with minor restrictions than the ones usually requested. In particular, our analysis covers some instance of shape regular meshes with edges arbitrarily short with respect to the diameters of the elements they belong to.
  \item Different instances of stabilization forms $s_E(\cdot,\cdot)$. Among them, we provide a detailed analysis of both the standard choice presented in \cite{volley,hitchhikers}, and a new one proposed in \cite{wriggers}. In addition, it is worth remarking that a stability analysis for this latter choice could be developed using the tools of \cite{volley}. However, the resulting error bound would be sub-optimal, in contrast with the numerical evidences. Our new approach, instead, leads to establish error bounds in perfect accordance with the numerical tests.  We also show that the choice presented in \cite{wriggers}, can have superior robustness properties \emph{in the presence of ``small'' edges}.
 \end{enumerate}
 \item {\em The development of a stability result concerning the choice of $s_E(\cdot,\cdot)$ presented in \cite{volley,hitchhikers} that is valid under more general mesh assumptions}. Essentially, we prove that the stabilization term is equivalent to the $H^1$ seminorm, where one of the two equivalence constants logarithmically degenerates in presence of ``small'' edges.
 \item {\em An interesting result regarding the structure of $s_E(\cdot,\cdot)$}. More precisely, we prove that the internal term $s_E^\circ(\cdot,\cdot)$ can be dropped without any detriment to the stability features of the underlying VEM scheme.
 \end{itemize}

A brief outline of the paper is as follows. We present the continuous model problem and we review its virtual element discretization in Section \ref{sec:1}.
In Section \ref{sec:techres} we develop a set of basic technical lemmas concerning virtual elements and polygonal elements. In Section \ref{sec:2} we present and develop our error analysis strategy. Afterwards, in Section \ref{sec:3} we apply such an approach in order to analyse some existing choices of the stability form, under more general mesh assumptions than the ones typically adopted in the VEM literature.
Finally, we present some numerical tests in Section \ref{sec:4}.
}

% ----------------------------------------------------
\section{The continuous and discrete problems}\label{sec:1}
% ----------------------------------------------------

In this section we briefly present the continuous problem and its discretization with the Virtual Element Method. More details can be found in \cite{volley,hitchhikers}.

% ----------------------------------------------------
\subsection{The continuous problem}
% ----------------------------------------------------

As a model elliptic problem we consider the diffusion problem in primal form.
Defining $(\cdot,\cdot)$ as the scalar product in $L^2$,
and $a(u,v):=(K \nabla u, \nabla v)$, the variational formulation of the problem reads:
\begin{equation}\label{cont-pbl}
\left\lbrace{
\begin{aligned}
&\mbox{Find } u\in V:=H^1_0(\O)~\mbox{such that}\\
&a(u,v)=(f,v)\quad \forall v\in V,
\end{aligned}
} \right.
\end{equation}
where $\O\subset \R^2$ is a polygonal domain and the loading $f \in L^2(\O)$.
The diffusion symmetric tensor $K=K(x,y)$ is assumed to satisfy:
$$
c|\ffxi|^2 \le \ffxi\cdot K(x,y)\ffxi \le C |\ffxi|^2 \qquad \forall\ffxi \in \R^2, \quad
\forall (x,y)\in \O .
$$
Above, $|\cdot|$ denotes the euclidean norm in $\R^2$.

It is well known that problem \eqref{cont-pbl} has a
unique solution, because our assumptions on $K$
and the Poincar\'e inequality yield:
\begin{equation}\label{ellipt-a}
 a(u,v)\le M\, |u|_{H^1(\O)}|v|_{H^1(\O)},\qquad a(v,v)\ge\,\alpha\,||v||^2_{H^1(\O)}\qquad \forall u,\,v\in V,
\end{equation}
with $0 < \alpha < M < 1$.

Note that the bilinear form $a(\cdot,\cdot)$ in~\eqref{ellipt-a} can obviously be split as
\begin{equation}\label{dec_a}
 a(v,w)=\sum_{{\E\in \Th}}a_E(v,w) \quad \textrm{ with } \quad
a_E(v,w) : = \int_E \nabla v \cdot \nabla w
\end{equation}
for all $v,w \in V$.

% ----------------------------------------------------
\subsection{The virtual element method}
% ----------------------------------------------------

Let an integer $k$, equal or greater than 1, and let $\{ \Omega_h \}_h$ denote a family of meshes, made of general simple polygons, on $\Omega$.
Given an element $E \in \Omega_h$ of diameter $h_E$ and area $|E|$, its boundary $\partial E$ is subdivided into $N=N(E)$ straight segments, which are called \emph{edges}, with a little abuse of terminology. Accordingly, the endpoints of the edges are called \emph{vertices} of the element $E$.
We remark that several consecutive edges of $E$ may be collinear; as a consequence, the number of edges (and vertices) may be greater than the number of maximal straight segments of $\partial E$. Hence, a triangle may have ten edges, for instance.
Furthermore, the length of an edge $e \in \partial E$ is denoted by $h_e$.
Moreover, in the sequel we assume that the diffusion tensor $K$ is piecewise constant
with respect to the meshes $\{ \Omega_h \}_h$.

For each $E \in \Omega_h$ we now introduce the local virtual space
$$
V_E = \big\{ v \in H^1(E) \cap C^0(E)\: : \: - \Delta v \in \Pp_{k-2}(E) \: , \ v|_e \in \Pp_k(e) \ \forall e \in \partial E \big\} ,
$$
where $\Pp_n$, $n \in {\mathbb N}$, denotes the polynomial space of degree $n$, $n \in {\mathbb N}$, with the convention that $\Pp_{-1} = \{ 0 \}$.
The associated set of local degrees of freedom $\Xi$ (divided into boundary ones $\Xi^\partial$, and internal ones $\Xi^\circ$) are given by
\begin{itemize}
\item point values at the vertexes of $E$;
\item for each edge, point values at $(k-1)$ distinct points on the edge (this are typically taken as Gauss-Lobatto nodes, see \cite{volley,hitchhikers});
\item the internal moments against a scaled polynomial basis $\{ m_i \}_{i=1}^{k(k-1)/2}$ of $\Pp_{k-2}(E)$
\begin{equation}\label{mom-dof}
\Xi_i^\circ(v) = |E|^{-1} \int_E v \: m_i \ , \quad \textrm{span}\{m_i\}_{i=1}^{k(k-1)/2} \!\! = \mathbb{P}_{k-2}(E) , \ \ || m_i ||_{L^\infty(E)} \simeq 1 .
\end{equation}
\end{itemize}
For future reference, we collect all the $Nk$ boundary degrees of freedom (the first two items above) and denote them with $\{\Xi_i^\partial\}_{i=1}^{Nk}$.

The global space $V_h \in H^1_0(\Omega)$ (such that $V_{h|E}=V_E$) is obtained by gluing the above spaces, and the same holds for the global degrees of freedom. We refer to \cite{volley} for the explicit expression.
On each element $E$ we also define a projector $\Pi^\nabla_E: V_E \rightarrow \Pp_k(E)$, orthogonal with respect to the bilinear form $a_E(\cdot,\cdot)$. More explicitly, for all $v \in V_E$:
\begin{equation}\label{proj}
\left\{
\begin{aligned}
& \Pi^\nabla_E v \in \Pp_k(E) \\
& a_E(v - \Pi^\nabla_E v, p) = 0 \quad \forall p \in \Pp_k(E) \\
& \PR(v - \Pi^\nabla_E v) = 0
% & \int_E (v - \Pi^\nabla_E v) = 0 ,
\end{aligned}
\right.
\end{equation}
where $\PR$ denotes any projection operator onto the space $\Pp_0(E)$.
{
% \begin{remark}\label{rem:zeroav}
In the literature one can find various choices for the operator $\PR$. The following three choices, which we focus on in the sequel, are among the most popular.

\begin{enumerate}

\item  A typical choice, that can be used for $k \ge 2$, is given by the average on the element $E$:
\begin{equation}\label{eq:Rint}
\PR v = |E|^{-1} \int_E v .
\end{equation}

\item An alternative, valid for any $k$, is to take the average on the boundary:
\begin{equation}\label{eq:Rbnd}
\PR v = |\partial E|^{-1} \int_{\partial E} v
\end{equation}

\item A third choice, again valid for any $k$, is the average of the vertex values:
\begin{equation}\label{eq:Rvarg}
\PR v = \frac{1}{N}\sum_{i=1}^N v(p_i),
\end{equation}
where the $p_i$'s denote the vertices of $E$.

\end{enumerate}
}

It is easy to check that the above projector $\Pi^\nabla_E$ is computable on the basis of the available degrees of freedom (see \cite{volley}).
Moreover, we introduce the following symmetric and positive semi-definite stability bilinear form on $V_E \times V_E$
\begin{equation}\label{st-term}
s_E(v,w) = s_E^\partial(v,w) + s_E^\circ(v,w) .
\end{equation}
Equation \eqref{st-term} highlights that $s_E$ is the sum of two contributions: the first, $s_E^\partial$, involving the boundary degrees of freedom; the second, $s_e^\circ$, involving the internal degrees of freedom. For instance, the \emph{standard choice} corresponds to:
\begin{equation} \label{stab-int-bound}
\begin{aligned}
& s_E^\partial(v,w)= \sum_{i=1}^{Nk} \Xi_i^\partial(v) \Xi_i^\partial(w)  \quad \textrm{(part involving the boundary DoFs)}, \\
& s_E^\circ(v,w)    =  \sum_{i=1}^{k(k-1)/2} \Xi_i^\circ(v) \Xi_i^\circ(w)  \quad \textrm{(part involving the internal DoFs)}.
\end{aligned}
\end{equation}

\begin{remark}\label{rem:interest}
Part of the interest of this paper is also the possibility to substitute $s_E^\partial(v,w)$ by some other option (non-standard choices). As an example, we will consider (cf. \cite{wriggers}):
\begin{equation}\label{optionalS}
s_E^\partial(v,w) = h_E \int_{\partial E} \partial_s v \: \partial_s w ,
\end{equation}
where $\partial_s$ denotes the tangent derivative along the edges.

Another interesting point considered in this paper, is the possibility to completely neglect
the internal part of the stability form. In other words, we will show that the choice
\begin{equation}\label{no-int-stab}
s_E^\circ(v,w)    =  0 ,
\end{equation}
does not spoil the stability feature of the numerical scheme.
\end{remark}

Given any symmetric and coercive form $s_E(\cdot,\cdot)$, one can define the local discrete bilinear forms on $V_E \times V_E$
\begin{equation}\label{ah-form}
a_E^h(v,w) = a_E(\Pi^\nabla_E v,\Pi^\nabla_E w) + s_E((I-\Pi^\nabla_E) v, (I-\Pi^\nabla_E)w) ,
\end{equation}
that are computable and approximate $a_E(\cdot,\cdot)$. Given the global discrete form
\begin{equation}\label{dec_ah}
a^h(v,w) = \sum_{E \in \Omega_h} a_E^h(v_{|E}, w_{|E}) \quad \forall v,w \in V_h ,
\end{equation}
the discrete problem is:
\begin{equation}\label{discr-pbl}
\left\{
\begin{aligned}
& \textrm{Find } u_h \in V_h \\
& a^h(u_h,v_h) = <f_h,v_h> \quad \forall v_h \in V_h .
\end{aligned}
\right.
\end{equation}
For a discussion about the approximated loading term $<f_h,v_h>$, we refer to \cite{volley,projectors}.

% \bigskip

In order to shorten the notation, and also to underline the generality of the proposed approach, in the following we will simply use $\Pi_E$ instead of $\Pi^\nabla_E$ to denote the projector operator.

The following assumptions on the mesh will be considered in the present work.
\begin{itemize}
\item[A1)] It exists $\gamma \in \mathbb{R}^+$ such that all elements $E$ of the mesh family $\{\Omega_h\}_h$ are star-shaped with respect to a ball $B_E$ of radius $\rho_E \ge \gamma h_E$ and center ${\bf x}_E$.
\item[A2)] It exists $C \in \mathbb{N}$ such that $N(E) \le C$ for all elements $E \in \{\Omega_h\}_h$.
\item[A3)] It exists $\eta \in \mathbb{R}^+$ such that for all elements $E$ of the mesh family $\{\Omega_h\}_h$ and all edges $e\in\partial E$ it holds $h_e \ge \eta h_E$.
\end{itemize}
In all the presented results we will explicitly write which of those hypotheses are used, if any.

We remark that assumptions A1 and A3 are those considered in \cite{volley}. Here, we want to consider also weaker assumptions in terms of the edge requirements, namely the combination of A1 and A2.
It is easy to check that, provided A1 holds, assumption A3 implies A2. However, assumption A2 is much weaker than A3, as it allows for edges arbitrarily small with respect to the element diameter.

In the following the symbol $\lesssim$ will denote a bound up to a constant that is uniform for all $E \in \{\Omega_h\}_h$ (but may depend on the polynomial degree $k$). Moreover, in order to make the notation shorter, for any non negative real $s$ we will denote by
$$
\| v \|_{s,\omega} = \| v \|_{H^s(\omega)} \ , \qquad | v |_{s,\omega} = | v |_{H^s(\omega)}
$$
the standard $H^s$ Sobolev (semi)norm on the measurable open set $\omega$.

%%
%Under assumption A1, we can build a sub-triangulation of each element $E \in \Omega_h$ by connecting each vertex of $E$ with the center of the associated ball. We denote by ${\cal T}_h$ the collection of all such triangles for a given mesh $\Omega_h$ in our family.
%It is easy to check that:
%\begin{itemize}
%\item[(i)] under assumption A1 all the triangles in the family $\{ {\cal T}_h \}_h$ have angles that are \emph{uniformly} bounded away from $\pi$;
%\end{itemize}

\begin{remark}\label{tau}
The stability form $s_E(\cdot,\cdot)$ may also be scaled by a multiplicative factor $\tau_E > 0$, to take into account the magnitude of the material parameter $K$, for instance. With this respect, a possible choice could be to set $\tau_E$ as the trace of $K$ on each element. In this paper, we do not investigate on how to select $\tau_E$, but we address the reader to \cite{VEM-elasticity,Paulino-VEM,BeiraoLovaMora} for some study on such an issue.
\end{remark}

% -------------------------------------------------------------------------------
\section{Preliminary results}\label{sec:techres}
% -------------------------------------------------------------------------------

In this section we present some technical results that will be needed in the sequel of the paper.

The $H^{1/2}$ boundary norm will have an important role in the following. We here use the (one-dimensional) double integral definition
\begin{equation}\label{H12def}
|v|_{1/2,\partial E}^2 := \int_{\partial E} \int_{\partial E} \left( \frac{v(s_1)-v(s_2)}{s_1-s_2} \right)^2 ds_1 ds_2 ,
\end{equation}
where, with a small abuse of notation, $v$ stands for $v|_{\partial E}$, and where $s_1,s_2$ denote curvilinear abscissae along the boundary.

% We start by showing some preliminary lemmas.
%%
We begin with the following Lemmas.
\begin{lem}\label{lem:tr:um}
Let assumption A1 hold. Then
\begin{equation}\label{B2}
|v|_{1/2,\partial E} \lesssim | v |_{1,E} \qquad \forall v \in H^1(E) , \ E \in \Omega_h.
\end{equation}
Moreover, for all $E \in \Omega_h$ and all $v \in H^{1/2}(\partial E)$, there exists an extension $\widetilde v \in H^1(E)$ such that
\begin{equation}\label{B2bis}
| \widetilde v |_{1,E} \lesssim |v|_{1/2,\partial E} .
\end{equation}
\end{lem}
\begin{proof}
We only sketch the simple proof, based on a mapping argument. Up to a translation of the element $E$, we may assume that the ball center ${\bf x}_E$ is the origin of the coordinate axes. Let then the function $\Psi : [0, 2\pi) \rightarrow [\rho_E, h_E]$ describe the boundary of $E$, as follows. The boundary curve $\Gamma=\partial E$ can be parametrized in a unique way as
\begin{equation}
\gamma(\theta) = \big( \Psi (\theta) \cos{(\theta)} , \Psi (\theta) \sin{(\theta)} \big) \ , \quad \theta \in [0, 2\pi) ,
\end{equation}
with $\theta$ representing the angle in radial coordinates. Note that property A1 implies $\Psi\in W^{1,\infty}[0, 2\pi)$, uniformly with respect to $E \in \Omega_h$.
We then introduce the radial mapping $F : \overline{B}_E \rightarrow \overline{E}$, associating a point expressed in polar coordinates
$$
\big( \hat x , \hat y \big) = \big( \hat r \cos{(\hat\theta)}, \hat r \sin{(\hat\theta)} \big) \ , \quad \hat r \in [0,\rho_E], \ \hat \theta \in [0, 2\pi) ,
$$
with the point $(x,y) = F(\hat x, \hat y)$, whose coordinates are
$$
\big( x , y \big)  = \big( r \cos{(\theta)}, r \sin{(\theta)} \big) \ , \quad r = \hat r \frac{\Psi(\hat\theta)}{\rho_E} , \ \theta=\hat\theta.
$$
By recalling A1, it can be checked that  $F\in W^{1,\infty}(B_E)$, and the same holds for the inverse mapping, i.e. $F^{-1} \in W^{1,\infty}(E)$. It is easy to see that $|F|_{1,\infty,B_E}\lesssim C$ and
$|F^{-1}|_{1,\infty,E}\lesssim C$.
%The associated norms are moreover uniformly bounded with respect to $E \in \Omega_h$.
%%
As a consequence, bound \eqref{B2} can be simply proved by a standard ``pull-back and push-forward'' argument: i) map $v \in H^1(E)$ from $E$ into $B_E$ using $F$; ii) notice that the trace bound analogous to \eqref{B2} holds on the ball $B_E$; iii) map back to $E$ using $F^{-1}$. Bound \eqref{B2bis} is similarly proved: one only needs to map the boundary data into $B_E$, to consider the harmonic extension inside $B_E$, and finally to map back to $E$.
\end{proof}

\begin{cor}\label{Hdiv-trace}
Let assumption A1 hold. Then
\begin{equation}\label{B2ter}
| {\bf w}\cdot {\bf n}_E |_{H^{-1/2}(\partial E)} \lesssim \| {\bf w} \|_{0,E} \qquad
\forall {\bf w} \in [L^2(E)]^2 \textrm{ with } \textrm{div}\, {\bf w}=0, \ \forall E \in \Omega_h ,
\end{equation}
with ${\bf n}_E$ denoting the outward unit normal to the boundary of $E$.
% and where as usual the squared $H_{div}$ norm is defined as the sum of the squared $L^2$ norms of the function and its divergence.
\end{cor}
\begin{proof}
By the definition of dual norm and using \eqref{B2bis}, we get
$$
\begin{aligned}
| {\bf w}\cdot {\bf n}_E |_{H^{-1/2}(\partial E)} & = \sup_{v \in H^{1/2}(\partial E)}
\frac{{}_{-1/2,\partial E}< {\bf w}\cdot {\bf n}_E , v >_{1/2,\partial E}}{|v|_{1/2,\partial E}} \\
& \lesssim \sup_{\widetilde v \in H^{1}(E)}
\frac{{}_{-1/2,\partial E}< {\bf w}\cdot {\bf n}_E , \widetilde v >_{1/2,\partial E}}{|\widetilde v|_{1,E}} .
\end{aligned}
$$
An integration by parts, using $\textrm{div}\,{\bf w}=0$, and the Cauchy-Schwarz inequality
lead to estimate \eqref{B2ter}:
$$
| {\bf w}\cdot {\bf n}_E |_{H^{-1/2}(\partial E)} \lesssim \sup_{\widetilde v \in H^{1}(E)}
\frac{\int_E {\bf w} \cdot \nabla \widetilde v}{|\widetilde v|_{1,E}}
\le \|  {\bf w} \|_{0,E} .
$$

\end{proof}

\begin{lem}\label{lem:lapl} Let assumption A1 hold true. Then we have:
\begin{equation}\label{eq:lapl}
\| \Delta v_h \|_{0,E} \lesssim h_E^{-1} | v_h |_{1,E} \quad \forall v_h \in V_E .
\end{equation}
\end{lem}
\begin{proof}
For $E \in \Omega_h$, let $T_E\subset E$ denote an equilateral triangle inscribed in the ball $B_E$. We start observing that, due to assumption A1, for any polynomial $p$ of given maximum degree it holds $ \| p \|_{0,E} \lesssim \| p \|_{0,T_E}$. This follows from noting that the smallest ball containing $E$ and the largest ball contained in $T_E$ have uniformly comparable radii.
We now recall that $\Delta v_h \in \Pp_{k-2}$. Let $b \in \Pp_3(T_E)$ denote the standard cubic bubble in $T_E$ with unitary maximum value. Standard properties and inverse estimates of polynomial spaces on shape regular triangles yield
$$
\begin{aligned}
\| \Delta v_h \|_{0,E}^2 & \lesssim \| \Delta v_h \|_{0,T_E}^2 \lesssim \int_{T_E} b \Delta v_h \Delta v_h  =
\int_{T_E} \nabla v_h \cdot \nabla (b \Delta v_h) \\
& \lesssim    | v_h |_{1,T_E}  h_E^{-1} \| b \Delta v_h \|_{0,T_E}
\le h_E^{-1}  | v_h |_{1,T_E}  \| \Delta v_h \|_{0,T_E} .
\end{aligned}
$$
Estimate \eqref{eq:lapl} now follows by observing $\| \Delta v_h \|_{0,T_E} \le \| \Delta v_h \|_{0,E}$.
\end{proof}

\begin{remark}\label{rm:vem-invest}
The same argument in the proof of Lemma \ref{lem:lapl} can be used to prove inverse estimates for polynomials of fixed maximum degree, on polygons satisfying assumption A1.
\end{remark}

The next Lemma can be considered as a variant of Lemma 3.1 in \cite{bertoluzza}, supposing
that the number of edges is uniformly bounded.

\begin{lem}\label{lem:A2}
Let A1 and A2 hold. For all $E \in \Omega_h$ and all $v_h \in V_E$ we have
$$
\| v_h \|_{L^{\infty}(\partial E)}^2 \lesssim \: ( h_E^{-1} \| v_h \|_{0,\partial E}^2 + |v_h|_{1/2,\partial E}^2) .
$$
\end{lem}
\begin{proof}
For $w_h \in V_E$, we recall that $w_{h|\partial E}\in C^0(\partial E)$ and $w_{h|\partial E}$ is a polynomial of degree at most $k$ on each edge.

In addition, we first suppose that $\int_{\partial E} w_h = 0$. Then, by definition \eqref{H12def} and by using a scaling argument on each edge of the mesh, we obtain
$$
|w_h|_{1/2,\partial E}^2 \ge \sum_{e \in \partial E} |w_h|_{1/2,e}^2 \gtrsim \sum_{e \in \partial E} \| \frac{\partial w_h}{\partial s}\|_{L^1(e)}^2 ,
$$
where $s$ denotes the curvilinear abscissae along the generic edge. By assumption A2 and recalling that $w_h$ is continuous on the boundary, from the above bound we have:
\begin{equation}\label{zeroav}
|w_h|_{1/2,\partial E}^2 \gtrsim \Big( \sum_{e \in \partial E} \| \frac{\partial w_h}{\partial s}\|_{L^1(e)} \Big)^2
= \| \frac{\partial w_h}{\partial s}\|_{L^1(\partial E)}^2 \ge \| w_h \|_{L^\infty(\partial E)}^2 ,
\end{equation}
where we also used that $w_h|_{\partial E}$ has zero average and thus it vanishes at least at one point of $\partial E$.
For a generic $v_h \in V_E$ (not necessarily with vanishing mean value), the proof follows easily from \eqref{zeroav} by adding and subtracting its average on the boundary $\overline{v}_h$ and simple bounds:
$$
\begin{aligned}
\| v_h \|_{L^\infty(\partial E)} & \le \| v_h-\overline{v}_h \|_{L^\infty(\partial E)} + \| \overline{v}_h \|_{L^\infty(\partial E)}
\lesssim | v_h-\overline{v}_h |_{1/2,\partial E} + |\partial E|^{-1/2} \| \overline{v}_h \|_{0, \partial E}  \\
& \lesssim | v_h |_{1/2,\partial E} + |\partial E|^{-1/2} \| v_h \|_{0,\partial E}
\lesssim | v_h |_{1/2,\partial E} + h_E^{-1/2} \| v_h \|_{0,\partial E} ,
\end{aligned}
$$
where $|\partial E|$ denotes the length of $\partial E$.
\end{proof}

The following approximation result is an extension of the one in \cite{VemSteklov} to the case of higher order norms and more general mesh assumptions.
\begin{thm}\label{thm:approx}
Let assumption A1 hold. Then there exists a real number $ \overline\sigma > 3/2 $ such that for all $u \in H^{s}(\Omega)$,
$1< s \le k+1 $, and all $1 \le \sigma < \min \{\overline\sigma, s\}$, it holds:
\begin{equation}\label{eq:chile-bis}
| u - u_I |_{\sigma,E} \lesssim h_E^{s-\sigma} |u|_{s,E} \ , \quad E \in \Omega_h,
% \ 1 \le \sigma < s \le k+1 ,
\end{equation}
where $u_I$ is the degrees-of-freedom interpolant of $u$ in $V_h$.
\end{thm}

\begin{proof}
For each element E, we can build a sub-triangulation by connecting all its vertexes with the center ${\bf x}_E$ introduced in assumption A1. We denote by ${\cal T}_h$ the global (conforming) triangular mesh obtained by applying such a procedure for all $E \in \Omega_h$.
It is easy to check that, under assumption A1, the triangles in the sequence of meshes $\{ {\cal T}_h \}_h$ have maximum angles that are uniformly bounded away from $\pi$ (although shape regularity is not guaranteed).

Let $u_r$ be the standard continuous and piecewise $\Pp_k$ polynomial Lagrange interpolant of $u$ over the triangulation ${\cal T}_h$.
Then it holds:
\begin{equation}\label{approx:fem:interp}
| u - u_r |_{\sigma,E} \lesssim h_E^{s-\sigma} |u|_{s,E} \ , \quad E \in \Omega_h, \ 1 \le \sigma < s \le k+1 , \ u \in H^{s}(\Omega) ,
\end{equation}
where we used the anisotropic approximation results in \cite{Apel-book}, also recalling the angle property above.
In the following we denote by $u_\pi$ a piecewise discontinuous polynomial approximation of $u$ over the mesh $\Omega_h$. For instance, one may think of the $L^2$ projection of $u$ on $\Pp_k(E)$ for each element $E$.

We now introduce the function $u_I \in V_h$ defined, on each element $E$, by
\begin{equation*}
\left\{
\begin{aligned}
& - \Delta u_I = - \Delta u_\pi \qquad \textrm{in } E , \\
& u_I = u_r \qquad \textrm{on } \partial E ,
\end{aligned}
\right.
\end{equation*}
so that $(u_I - u_\pi)$ satisfies on every $E$
\begin{equation}\label{diff-probl}
\left\{
\begin{aligned}
& - \Delta (u_I -u_\pi) = 0 \qquad \textrm{in } E , \\
& u_I - u_\pi = u_r - u_\pi \qquad \textrm{on } \partial E .
\end{aligned}
\right.
\end{equation}
Therefore, for all $E \in \Omega_h$, regularity results on Lipschitz domains (see \cite{grisvard}) guarantee that
\begin{equation}\label{eq:app:reg}
| u_I -u_\pi |_{\sigma,E} \lesssim | u_r - u_\pi |_{\sigma-1/2,\partial E} \quad 1 \le \sigma \le \overline{\sigma}_E ,
\end{equation}
where $\overline{\sigma}_E=2$ if $E$ is convex, and $\overline{\sigma}_E = 1 + \pi / \omega_E$ (with $\omega_E$ the largest angle of $E$) otherwise.
Let now $\overline{\sigma} = \min_{E \in \{ \Omega_h\}_h} \overline{\sigma}_E$, where we stress that the minimum is taken among all elements of the whole mesh sequence.
Due to assumption $A1$, that yields a uniform bound on the maximum element angles, the number $\overline{\sigma}$ is \emph{strictly} bigger than $3/2$.
First a triangle inequality and bound \eqref{eq:app:reg}, then a trace inequality yield
$$
\begin{aligned}
| u - u_I |_{\sigma,E} & \lesssim
| u - u_\pi |_{\sigma,E} + | u_r - u_\pi |_{\sigma-1/2,\partial E} \lesssim
| u - u_\pi |_{\sigma,E} + | u_r - u_\pi |_{\sigma,E}  \\
& \lesssim | u - u_\pi |_{\sigma,E} + | u - u_r |_{\sigma,E}
\end{aligned}
$$
for all $1 \le \sigma \le \overline{\sigma}$ and all $E \in \Omega_h$.
The result follows combining the above bound with \eqref{approx:fem:interp} and standard polynomial approximation estimates on shape regular polygons.

\end{proof}

% \end{remark}

{

Regarding the operator $\PR$, we have the following Lemma.

\begin{lem}\label{lm:newlemma}
Let A1 hold. For the operators $\PR$ described in \eqref{eq:Rint}, \eqref{eq:Rbnd} and \eqref{eq:Rvarg}, we have the following approximation properties.

\begin{enumerate}

\item For $\PR$ defined by \eqref{eq:Rint} or by \eqref{eq:Rbnd}:
\begin{equation}\label{R-approx}
\| v - \PR v \|_{0,E} \lesssim h_E |v|_{1,E} \quad \forall v \in H^{1}(E) .
\end{equation}

\item For $\PR$ defined by \eqref{eq:Rvarg}, more regularity for $v$ is needed, namely
\begin{equation}\label{R-approxh1}
\mbox{if $\varepsilon >0$}\quad \| v - \PR v \|_{0,E} \lesssim h_E |v|_{1,E} + h_E^{1+\varepsilon} |v|_{1+\varepsilon,E} \quad \forall v \in H^{1+\varepsilon}(E) ,
\end{equation}
unless $v$ is a polynomial, in which case it holds:
\begin{equation}\label{R-approx-pol}
\| p - \PR p \|_{0,E} \lesssim h_E |p|_{1,E} \quad \forall p \in \Pp_k(E) .
\end{equation}
Finally, if also assumption A2 holds, then
\begin{equation}\label{R-approxh2}
\| v_h - \PR v_h \|_{0,E} \lesssim h_E |v_h|_{1,E} \quad \forall v_h \in V_E  .
\end{equation}

\end{enumerate}

% From the above bound, an inverse estimate on polynomials easily yields
% \begin{equation}\label{R-approx-p}
% \| p - \PR p \|_{L^2(E)} \lesssim h_E |p|_{H^1(E)} \quad \forall p \in\Pp_k(E).
% \end{equation}

\end{lem}

\begin{proof}
Estimates \eqref{R-approx} and \eqref{R-approxh1} follow, recalling assumption A1, from standard approximation theory on shape regular polygons.
Bound \eqref{R-approx-pol} follows immediately from \eqref{R-approxh1} by using an inverse estimate for polynomials on polygons, see Remark \ref{rm:vem-invest}.

To prove \eqref{R-approxh2}, take any $v_h\in V_E$ and set
$\overline{v}_h :=|\partial E|^{-1} \int_{\partial E} v_h$. We have
\begin{equation}\label{vmean-1}
||v_h -\PR v_h ||_{0,E} = || (v_h - \overline{v}_h) + \PR(\overline{v}_h - v_h) ||_{0,E}
\le || v_h - \overline{v}_h ||_{0,E} + || \PR(\overline{v}_h - v_h) ||_{0,E}
\end{equation}
From \eqref{R-approx} we get
\begin{equation}\label{vmean-2}
|| v_h - \overline{v}_h ||_{0,E} \lesssim h_E |v_h|_{1,E} .
\end{equation}
Furthermore, from \eqref{eq:Rvarg} and recalling that $\PR v_h$ is a constant, we get
\begin{equation}\label{vmean-3}
|| \PR(\overline{v}_h - v_h) ||_{0,E} \lesssim h_E || \PR(\overline{v}_h - v_h) ||_{L^\infty(\partial E)}
\lesssim h_E || \overline{v}_h - v_h ||_{L^\infty(\partial E)} .
\end{equation}
Since $\overline{v}_h - v_h$ has zero mean value on $\partial E$, using Lemma \ref{lem:A2} we get
$$
|| \overline{v}_h - v_h ||_{L^\infty(\partial E)} \lesssim | v_h|_{1/2,\partial E} .
$$
Hence, from \eqref{vmean-3} and \eqref{B2}, we obtain
\begin{equation}\label{vmean-4}
|| \PR(\overline{v}_h - v_h) ||_{0,E} \lesssim h_E |  v_h |_{1/2,\partial E}
\lesssim h_E |  v_h |_{1,E} .
\end{equation}
Estimates \eqref{vmean-1}, \eqref{vmean-2} and \eqref{vmean-4} give \eqref{R-approxh2}.

\end{proof}

}

% ------------------------------------------------------
\section{A general error analysis}\label{sec:2}
% ------------------------------------------------------

In the present section we derive an error analysis
which is more general than the standard one detailed in \cite{volley}. We remark that the present approach
can be applied to any other linear symmetric elliptic problem.

For the analysis, the following discrete semi-norm, induced by the stability term, will play an important role:
\begin{equation}\label{st-seminorm}
\tri v \tri_E ^2 := s_E \big( (I-\PR) v , (I-\PR) v \big)
+ a_E(\Pi_E v , \Pi_E v)\quad \forall v \in V_{h|\E} + {\cal V}_{\E}.
\end{equation}
Above, ${\cal V}_{\E} \subseteq V_{|\E}$ is a subspace of sufficiently regular functions in order for $s_E(\cdot,\cdot)$ to make sense.

We now introduce the following assumption, for all $\E\in\Th$.

\smallskip
{\bf Main assumption} - We assume that it holds
\begin{eqnarray}\label{ass-1}
&& a_E(v_h,v_h) \le C_1(E) \: \tri v_h \tri_{\E}^2 \qquad \forall v_h \in V_{h|\E} , \\
\label{ass-2}
&&  \tri p \tri_{\E}^2 \le C_2(E) \: a_E(p,p) \qquad \forall p \in \Pp_k(\E) ,
\end{eqnarray}
with $C_1(E),\ C_2(E)$ positive constants which depend on the shape and possibly on the size of $E$.

\begin{lem}\label{lemma-1}
Under assumptions \eqref{ass-1}, \eqref{ass-2}, the local discrete bilinear form \eqref{ah-form} satisfies the stability condition
\begin{equation}\label{lemmabound1}
C_\star(E)\tri v_h \tri_{\E}^2
\lesssim a_E^h (v_h,v_h)
\lesssim C^\star(E) \tri v_h \tri_{\E}^2
\qquad \forall v_h \in V_{h|\E} ,
\end{equation}
and also the bound
\begin{equation}\label{lemmabound2}
a_E^h (v_h,v_h) \lesssim \big(1+C_2(E) \big) \big( \tri v_h \tri_{\E}^2 + | v_h |_{1,E}^2 \big)
\quad \forall v_h \in V_{h|\E} ,
\end{equation}
where $\displaystyle{C_\star(E)= {\rm min} \{ 1 , C_2(E)^{-1} \}}$
and $C^\star(E)={\rm max}(1, C_1(E)C_2(E))$.
\end{lem}
\begin{proof}
%Let, for any $q \in \Pp_k(\E)$, the decomposition
%$q = q^\bot + \overline{q}$ with $q^\bot \in \Pp_k^\bot(\E)$
%and $\overline{q} \in \Pp_0(\E)$. Note that $\wP_\E(\overline{q})=0$
%by definition of the operator and thus $\wP_\E(q)=q^\bot$.
%%%
%Then, from \eqref{ah-form} and the definition of $\wP_\E$ it follows, for all $q \in \Pp_k(\E)$ and $v_h \in V_{h|\E}$,
%$$
%a_E^h (v_h,q) = a_E (\wP_\E v_h , q^\bot) + \scl v_h - \wP_\E v_h , \overline{q} \scr_\E
%=  a_E (v_h , q) ,
%$$
%where in the last step we used that $\overline{q}$
%is in the kernel of the bilinear form $a_E$
%and, due to \eqref{ass-2}, also of the form $\scl \cdot , \cdot \scr_{\E}$.
%%%
%The equation above is the consistency condition.
%
We start by noting that, from definition \eqref{proj} it is immediate to check that
\begin{equation}\label{eq:ovvio}
\PR (I - \Pi_E) v_h = 0 \quad \forall v_h \in V_{h|\E} .
\end{equation}
Using first \eqref{eq:ovvio}, then noting that $\Pi_E(I - \Pi_E)=0$ and applying \eqref{ass-2}, we obtain (cf. \eqref{st-seminorm})
\begin{equation}\label{eq:lem:1}
\begin{aligned}
 a_E^h (v_h,v_h) & = a_E (\Pi_E v_h, \Pi_E v_h) + s_E((I-\PR) (I-\Pi_E)v_h , (I-\PR)(I-\Pi_E)v_h) \\
&= a_E (\Pi_E v_h, \Pi_E v_h) + a_E (\Pi_E(I - \Pi_E)v_h, \Pi_E(I - \Pi_E)v_h)\\
& + s_E((I-\PR) (I-\Pi_E)v_h , (I-\PR)(I-\Pi_E)v_h)\\
&  \ge C_2(E)^{-1} \tri \wP_\E v_h \tri_\E^2
+ \tri v_h - \wP_\E v_h \tri_\E^2 \ge C_\star(E) \tri v_h \tri_{\E}^2
\end{aligned}
\end{equation}
for all $v_h \in V_{h|\E}$, with $\displaystyle{C_\star(E) = \frac{1}{2}{\rm min} \{ 1 , C_2(E)^{-1} \}}$.
Again using the first identity in \eqref{eq:lem:1}, recalling definition \eqref{st-seminorm}, from the triangle inequality we get
%
% that $\Pi_E$ is a projection with respect to $a_E(\cdot,\cdot)$ yields
% $$
% a_E^h (v_h,v_h) \le a_E (v_h, v_h) + \tri v_h - \wP_\E v_h \tri_\E^2 .
% $$
%
% From the above bound, using \eqref{ass-1} and a triangle inequality it is immediate to derive
%
\begin{equation}
a_E^h (v_h,v_h) \le  \tri v_h \tri_{\E}^2 + \tri v_h - \wP_\E v_h \tri_\E^2 \le
3 \tri v_h \tri_{\E}^2 +  2 \tri \wP_\E v_h \tri_{\E}^2  .
\end{equation}
Since $\wP_\E$ is a projection with respect to $a^\E$
and using \eqref{ass-2} we obtain
\begin{equation}\label{X-L-1}
\tri \wP_\E v_h \tri_{\E}^2 \le C_2(E) \: a_E(\wP_\E v_h,\wP_\E v_h) \le
C_2(E) \: a_E(v_h,v_h) .
\end{equation}
From \eqref{X-L-1} we immediately get
$$
\tri \wP_\E v_h \tri_{\E}^2 \le C_2(E) \: M \: | v_h |_{1,E}^2 ,
$$
and also, recalling \eqref{ass-1},
$$
\tri \wP_\E v_h \tri_{\E}^2 \le C_1(E) C_2(E) \: \tri v_h \tri_{\E}^2 .
$$
Combining the above bounds it follows
$$
\begin{aligned}
& a_E^h (v_h,v_h) \le 3 \tri v_h \tri_{\E}^2 + 2 C_2(E) \: M | v_h |_{1,E}^2 , \\
& a_E^h (v_h,v_h) \le C^\star(E) \tri v_h \tri_{\E}^2
\end{aligned}
$$
with $C^\star(E)=3 + 2C_1(E)C_2(E) $.
\end{proof}

As an immediate consequence of Lemma \ref{lemma-1} and \eqref{ass-1},
the discrete bilinear form \eqref{dec_ah} associated
to \eqref{ah-form} satisfies
\begin{equation}\label{discr-stab}
a^h (v_h,v_h) \ge C_{stab}(h) a (v_h,v_h) \ge C_{stab}(h) \alpha \:|| v_h ||_{H^1(\O)}^2
\qquad \forall v_h \in V_{h} ,
\end{equation}
where

\begin{equation}\label{cstab}
C_{stab}(h) = \displaystyle{ \min_{E\in \Th}\frac{C_\star(E)}{C_1(E)} }.
\end{equation}
Therefore, due to \eqref{ellipt-a}, the discrete problem
is positive definite and problem \eqref{discr-pbl} has a unique solution.

We have moreover the following convergence result.
For all sufficiently regular functions $v$, introduce the global semi-norms
\begin{equation}\label{st-norm}
\tri v \tri^2 = \sum_{\E\in\Th} \tri v \tri_{\E}^2 \ , \qquad | v |^2_{1,h} = \sum_{\E\in\Th} | v |_{1,E}^2 .
\end{equation}
We notice that, by~\eqref{ass-1} and the Poincar\'e inequality,
$\tri\cdot  \tri$ is a norm on $V_h$, not only a semi-norm.
Furthermore, for any $h$, let $\mathfrak{F}_h$
denote the quantity
\begin{equation}\label{defrhs}
\mathfrak{F}_h = \sup_{v\in V_h}\frac{(f,v)-<f_h,v>}{\tri v\tri} .
\end{equation}
We remark that, again by~\eqref{ass-1}, it holds:
\begin{equation}\label{defrhs-bis}
\mathfrak{F}_h \lesssim \sup_{v\in V_h}\frac{(f,v)-<f_h,v>}{|v|_{1,\Omega}} .
\end{equation}
Therefore, taking $f_h$ as in \cite{volley} and using the arguments in that paper, we infer:
\begin{equation}\label{defrhs-ter}
\mathfrak{F}_h \leq C(f) h^k ,
\end{equation}
where $C(f)$ depends on suitable Sobolev norms of the source term $f$.
%
%
% \begin{equation}\label{defrhs}
% (f,v)-<f_h,v>\,\le \,\mathfrak{F}_h\,| v|_{H^1(\O)}\quad\forall\, v\in V_h.
% \end{equation}
%

We have the following result.

\begin{thm}\label{teoconv}
Let assumptions \eqref{ass-1}-\eqref{ass-2} hold and let
the continuous solution of \eqref{cont-pbl} satisfy $u_{|\E} \in {\cal V}_\E$ for all $\E \in \Th$.
Then, for every $u_I\in V_h$ and for
every $\up$ such that $u_{\pi|E}\in\Pp_k(E)$, the discrete solution $u_h$
of \eqref{discr-pbl} with bilinear form \eqref{ah-form} satisfies
\begin{equation}\label{teoconv:eq}
| u-u_h |_{H^1(\Omega)} \,\lesssim C_{err}(h) \Big(
(\mathfrak{F}_h) + \tri u-u_I \tri + \tri u-\up \tri
+ | u-u_I |_{H^1(\Omega)} + |u-\up|_{1,h}
\Big).
\end{equation}
Setting
\begin{equation}\label{cnsts}
\widetilde C(h) =\max_{E\in\Th} \{1, C_2(E) \}   , \quad
C_1(h) =\max_{E\in\Th} \{C_1(E) \}  , \quad
C^\star(h) =\max_{E\in\Th} \{C^\star(E) \} ,
\end{equation}
the constant $C_{err}(h)$ is given by
$C_{err}(h) = \max {\{ 1, \widetilde C(h) C_1(h), \widetilde C(h)^{3/2} \sqrt{C^\star(h) C_1(h)}  \}} $.
\end{thm}
\begin{proof}
First using the coercivity property in Lemma \ref{lemma-1},
then with identical calculations as in Theorem 3.1,
equation (3.11), of \cite{volley}, we get
\begin{equation}\label{EQ1}
\tri u_h-u_I \tri^2 \le \widetilde C(h) a^h(u_h-u_I,u_h-u_I) =
\widetilde C(h) \left( T_1 + T_2 + T_3 \right),
\end{equation}
where $\displaystyle{ \widetilde C(h) =\max_{E\in\Th} \{1, C_2(E) \}}$, and the terms $T_i$ are given by
$$
\begin{aligned}
& T_1 = <f_h,u_h-u_I>-(f,u_h-u_I) , \\
& T_2 = \sum_{\E\in\Th} a_E^h(\up-u_I,u_h-u_I) , \\
& T_3 = \sum_{\E\in\Th} a_E(u-\up,u_h-u_I) .
\end{aligned}
$$
For term $T_1$, definition \eqref{defrhs} and assumption \eqref{ass-1} yield
\begin{equation}\label{T1est}
T_1 \lesssim \mathfrak{F}_h | u_h-u_I |_{1,\O} \lesssim \sqrt{C_1(h)}\,
\mathfrak{F}_h \tri u_h-u_I \tri ,
\end{equation}
where $\displaystyle{ C_1(h) =\max_{E\in\Th} \{C_1(E) \}}$.
Term $T_2$ is treated using \emph{both} the bounds \eqref{lemmabound1} and \eqref{lemmabound2}, that easily lead to the estimate
\begin{equation}\label{T2est}
\begin{aligned}
T_2 & \lesssim \sqrt{C^\star(h) \widetilde C(h)} \Big( \tri \up-u_I \tri + | \up-u_I |_{1,\Omega} \Big)
\tri u_h-u_I \tri \\
& \le \sqrt{C^\star(h)\widetilde C(h)} \Big( \tri u-u_I \tri + \tri u-\up \tri + | u-u_I |_{1,\Omega} + | u-\up |_{1,h} \Big)
\tri u_h-u_I \tri ,
\end{aligned}
\end{equation}
where $\displaystyle{ C^\star(h) =\max_{E\in\Th} \{C^\star(E) \}}$.
Term $T_3$ is bounded using the piecewise continuity in
$H^1$ of the continuous bilinear form and \eqref{ass-1}
\begin{equation}\label{T3est}
\begin{aligned}
T_3
\lesssim  \sqrt{C_1(h)} \sum_{\E\in\Th} |u-\up|_{1,E} \tri u_h-u_I \tri_{\E}
\le \sqrt{C_1(h)} \: | u-\up |_{1,h} \tri u_h - u_I \tri .
\end{aligned}
\end{equation}

From \eqref{EQ1}, using the bounds \eqref{T1est}, \eqref{T2est} and \eqref{T3est},
then dividing by $\tri u_h-u_I \tri$, we get
\begin{equation}\label{tri-final}
\begin{aligned}
\tri u_h-u_I \tri & \lesssim \widetilde C(h) \max {\{ \sqrt{C_1(h)},\sqrt{C^\star(h)\widetilde C(h)} \}}  \\
&  \times \Big( \mathfrak{F}_h + \tri u-u_I \tri + \tri u-\up \tri + | u-u_I |_{1,\Omega}
+ | u-\up |_{1,h} \Big) .
\end{aligned}
\end{equation}
The triangle inequality and \eqref{ass-1} give
\begin{equation}\label{tri-final2}
| u-u_h |_{H^1(\Omega)} \le | u-u_I |_{H^1(\Omega)} + | u_h-u_I |_{H^1(\Omega)}
\le | u-u_I |_{H^1(\Omega)} + \sqrt{C_1(h)} \tri u_h-u_I \tri ,
\end{equation}
Combining \eqref{tri-final} and \eqref{tri-final2}, we get \eqref{teoconv:eq} with
$$
C_{err}(h) = \max {\{ 1, \widetilde C(h) C_1(h), \widetilde C(h)^{3/2} \sqrt{C^\star(h) C_1(h)}  \}} .
$$

\end{proof}

\begin{remark}\label{rem:norm-est}
By using \eqref{tri-final} and the triangle inequality it is immediate to check that, as a corollary of the above result, it also holds
$$
\tri u-u_h \tri \,\leq \widetilde C_{err}(h)\Big(
(\mathfrak{F}_h) + \tri u-u_I \tri + \tri u-\up \tri
+  |u-\up|_{1,h}
\Big) .
$$
where $\widetilde C_{err}(h)=\max {\{1,  \widetilde C(h)\sqrt{C_1(h)},  \widetilde C(h)^{3/2} \sqrt{C^\star(h)} \}} $.
\end{remark}

%\begin{remark}\label{rem:genmesh}
%We notice that Theorem~\ref{teoconv} still holds
%true without assuming any of the mesh conditions previously introduced, it is sufficient that each
%element $\E$ is a {\em simple polygon}.
%\end{remark}

% ------------------------------------------------------
\subsection{Reduction to the boundary}
\label{ssec:boundaryred}
% ------------------------------------------------------

In the present section we derive a result that allows to focus the analysis of assumptions \eqref{ass-1} and \eqref{ass-2} only on the boundary of the element. We here consider two cases for the internal stabilization form $s_E^\circ (\cdot,\cdot)$:
\begin{enumerate}
 \item as in the standard VEM (e.g. \cite{volley}), we put (see \eqref{stab-int-bound})

\begin{equation*}
 s_E^\circ(v,w)    =  \sum_{i=1}^{k(k-1)/2} \Xi_i^\circ(v) \Xi_i^\circ(w)  \quad \textrm{(part involving the internal DoFs)} ;
\end{equation*}

 \item we completely neglect the internal contribution (see \eqref{no-int-stab}), i.e.

\begin{equation*}
s_E^\circ(v,w)    =  0 .
\end{equation*}

\end{enumerate}

Both cases will lead to the same results.
Instead, the boundary bilinear form $s_E^\partial (\cdot,\cdot)$ is left completely general for the moment, and different choices will be made and analysed in Section \ref{sec:3}.

We start by showing the following Lemma.

\begin{lem}\label{pol-delta}
For all $v_h \in V_E$, there exists a polynomial $\widetilde p \in \Pp_k(E)$ such that $\Delta \widetilde p = \Delta v_h$ satisfying:
\begin{equation}\label{eq:pol-delta}
| \widetilde p |_{1,E} \lesssim h_E || \Delta v_h ||_{1,E} .
\end{equation}
\end{lem}
\begin{proof}
We only sketch the very simple proof. Since for all $v_h \in V_E$ it holds $\Delta v_h \in \Pp_{k-2}(E)$, there are (infinitely many) polynomials of degree $k$ that satisfy $\Delta \widetilde p = \Delta v_h$ (cf. \cite{SL-Sobolev}, for instance). In order to derive the bound \eqref{eq:pol-delta}, we first note that, thanks to assumption A1 and since $\Delta \widetilde p = \Delta v_h$, inequality \eqref{eq:pol-delta} is equivalent to
$$
| \widetilde p |_{1,B_E} \lesssim h_E || \Delta \widetilde p ||_{1,B_E} .
$$
The above bound, that is now restricted on balls, can be easily deduced by choosing $\widetilde p$ in the subspace
$$
\big\{ q \in \Pp_{k}(B_E) \ : \ \int_{B_E} \!\! q \: p = 0 \ \textrm{for all harmonic polynomials } p \in \Pp_k(B_E) \big\}
$$
and by a scaling argument.
\end{proof}

Concerning assumption \eqref{ass-1}, we have the following result.

\begin{prop}
\label{prop:upperbound}
Let assumption A1 and let $s_E^\circ (\cdot,\cdot)$ be given as in \eqref{stab-int-bound} or \eqref{no-int-stab}. Assume the existence of a positive constant $\widehat C_1(E)$ such that
\begin{equation}\label{eq:hyp:1}
\begin{aligned}
| v_h |_{1/2,\partial E}^2 % + h_E^{-1} \| (I-\PR) v_h \|_{0,\partial E}^2
\le \widehat C_1(E) \: \Big( s_E^\partial\big( (I-\PR) v_h , (I-\PR) v_h \big)&  + |\Pi_E v_h|_{1,E}^2 \Big)\\
&\forall v_h \in V_E .
\end{aligned}
\end{equation}
Then assumption \eqref{ass-1} holds with $C_1(E) \lesssim \max{\{ 1 , \widehat C_1 (E)\}}$.
\end{prop}
\begin{proof}
Let $v_h \in V_E$ and $\widetilde p$ as in Lemma \ref{pol-delta}. Let $\overline v_h$ be the unique constant function on $E$ such that $\int_{\partial E} \overline v_h = \int_{\partial E}  v_h$.
Then, first by an integration by parts and then by the definition of $\Pi_E$, we get
\begin{equation}\label{use-11}
\begin{aligned}
\int_E (v_h - \overline v_h) \Delta v_h \: dx & = \int_E (v_h - \overline v_h) \Delta \widetilde p \: dx
= - \int_E \nabla v_h \cdot \nabla \widetilde p  \: dx + \int_{\partial E} (v_h - \overline v_h) (\nabla \widetilde p \cdot {\bf n}_E ) \: ds \\
& = - \int_E \nabla \Pi_E v_h \cdot \nabla \widetilde p  \: dx + \int_{\partial E} (v_h - \overline v_h) (\nabla \widetilde p \cdot {\bf n}_E ) .
\end{aligned}
\end{equation}
Again an integration by parts and \eqref{use-11} yield
\begin{equation}\label{general}
\begin{aligned}
a_E(v_h,v_h) & \lesssim |v_h|_{1,E}^2 = | (v_h - \overline v_h) |_{1,E}^2 \\
& = - \int_{E} (v_h - \overline v_h) \Delta v_h \: dx \: + \: \int_{\partial E}  (v_h - \overline v_h) (\nabla v_h \cdot {\bf n}_E ) \: ds \\
& =  \int_E \nabla \Pi_E v_h \cdot \nabla \widetilde p  \: dx
+ \: \int_{\partial E}  (v_h - \overline v_h) (\nabla (v_h - \widetilde p) \cdot {\bf n}_E ) \: ds \\
& = T_1 + T_2
\end{aligned}
\end{equation}
with ${\bf n}_E$ denoting the outward unit normal to the boundary of $E$.
The first term above is bounded by the Cauchy-Schwarz inequality, Lemma \ref{pol-delta} and Lemma \ref{lem:lapl}. We obtain
\begin{equation}\label{use-14}
T_1  \le | \Pi_E v_h |_{1,E} \: | \widetilde p |_{1,E} \lesssim |\Pi_E v_h|_{1,E} | v_h |_{1,E} \lesssim \tri v_h \tri_E  | v_h |_{1,E} .
\end{equation}
For the second term, we first note that $\textrm{div} (\nabla (v_h - \widetilde p)) = \Delta (v_h - \widetilde p)$ = 0.
Therefore, after applying a (scaled) duality bound on the boundary of $E$, we can use Corollary \ref{Hdiv-trace} with ${\bf w}=\nabla (v_h - \widetilde p)$, and obtain
\begin{equation}\label{use-12}
\begin{aligned}
|T_2| & \lesssim
\Big( | (v_h - \overline v_h) |_{1/2,\partial E} + h_E^{-1/2} \| (v_h - \overline v_h) \|_{0,\partial E} \Big)
\| \nabla (v - \widetilde p) \cdot {\bf n}_E \|_{-1/2, \partial E} \\
& \lesssim \Big( | v_h |_{1/2,\partial E} + h_E^{-1/2} \| (v_h - \overline v_h) \|_{0,\partial E} \Big)
\: | (v_h - \widetilde p) |_{1,E} .
\end{aligned}
\end{equation}
Moreover, by standard approximation estimates in one dimension, it holds $h_E^{-1/2} \| (v_h - \overline v_h) \|_{0,\partial E} \lesssim | v_h |_{1/2,\partial E}$. Therefore,  using \eqref{eq:hyp:1}, the triangle inequality and again Lemmas \ref{pol-delta} and \ref{lem:lapl}, bound \eqref{use-12} yields
%% ----
\begin{equation}\label{II-bound}
\begin{aligned}
|T_2| & \lesssim  | v_h |_{1/2,\partial E} \big( | v_h |_{1,E} + | \widetilde p) |_{1,E} \big) \\
& \le \widehat C_1(E) \: \Big( s_E^\partial\big( (I-\PR) v_h , (I-\PR) v_h \big)  + |\Pi_E v_h|_{1,E}^2 \Big)
| v_h |_{1,E} \\
& \le \widehat C_1(E) \: \tri v_h \tri_E  | v_h |_{1,E} .
\end{aligned}
\end{equation}
The result follows by combining equations \eqref{general}, \eqref{use-14}, \eqref{II-bound} and recalling that $| v_h |_{1,E}^2 \lesssim a_E (v_h,v_h)$.
\end{proof}

Furthermore, concerning assumption \eqref{ass-2}, we have the following result.

\begin{prop}
\label{prop:lowerbound}
Let assumption A1 hold and let $s_E^\circ (\cdot,\cdot)$ as given in \eqref{stab-int-bound} or \eqref{no-int-stab}. Assume the existence of a positive constant $\widehat C_2(E)$ such that
\begin{equation}\label{eq:hyp:2}
s_E^\partial\big( (I-\PR) p , (I-\PR) p \big) \le \widehat C_2(E) \: | p |_{H^1(E)}^2
\quad \forall p \in \Pp_k(E) .
\end{equation}
Then assumption \eqref{ass-2} holds with $C_2(E) \lesssim \max{\{ 1 , \widehat C_2 (E)\}}$.
\end{prop}
\begin{proof} We first note that the second term in \eqref{st-seminorm} is immediately bounded:
\begin{equation}\label{trivest}
 a_E(\Pi_E p, \Pi_E p) = a_E( p, p) \le M  | p |_{H^1(E)}^2
\quad \forall p \in \Pp_k(E) .
\end{equation}
Therefore, by the definition of $s_E(\cdot,\cdot)$ and using \eqref{eq:hyp:2}, it is sufficient to show that
\begin{equation}\label{eq:int-pol-est}
s_E^\circ\big( (I-\PR) p , (I-\PR) p \big) \lesssim | p |_{H^1(E)}^2 .
\quad \forall p \in \Pp_k(E)
\end{equation}
Clearly, the above bound is trivial for the choice \eqref{no-int-stab}. Hence, we can focus on the choice \eqref{stab-int-bound}.
By definition of $s_E^\circ(\cdot,\cdot)$ and recalling that $\| m_i \|_{L^\infty(E)} \lesssim 1$, $i=1,2,...,n_{k-2}$, we have
\begin{equation}\label{X-newnew}
\begin{aligned}
& s_E^\circ\big( (I-\PR) p , (I-\PR) p \big) = \sum_i^{n_{k-2}} \Xi_i^\circ\big( (I-\PR)p \big)^2 \\
& = \sum_i^{n_{k-2}}  |E|^{-2} \Big( \int_E  \big((I-\PR) p \big) \ m_i \Big)^2
\lesssim \sum_i^{n_{k-2}}  |E|^{-1} \|  (I-\PR) p \|_{L^2(E)}^2 .
\end{aligned}
\end{equation}
Using property \eqref{R-approx} for the operator $\PR$ (or using \eqref{R-approx-pol} if the choice \eqref{eq:Rvarg} is being used), we now have
%Since $\PR$ is the average operator, see Remark \ref{rem:zeroav}, assumption A1 and standard approximation estimates on polygons yield
\begin{equation}
s_E^\circ\big( (I-\PR) p , (I-\PR) p \big)  \lesssim \sum_i^{n_{k-2}} |  p |_{H^1(E)}^2
\lesssim |  p |_{H^1(E)}^2 .
\end{equation}
\end{proof}

% -----------------------------------------------------------------------------------------------------------------------------
\section{Analysis of some choices for the boundary stabilization}
\label{sec:3}
% ------------------------------------------------------------------------------------------------------------------------------

In the present section we apply Propositions \ref{prop:upperbound} and \ref{prop:lowerbound} for a couple of standard choices of the boundary stability term $s^\partial_E(\cdot,\cdot)$. This allows to relax the mesh assumptions (with respect to the theory presented in \cite{volley}) in establishing stability and convergence properties of the proposed methods.

% ----------------------------------------------------------------------------------------------
\subsection{Identity matrix choice}
\label{ssec:3:1}
% ----------------------------------------------------------------------------------------------

This is the more standard, and simpler to code, choice for virtual elements. We recall it here again, for convenience:
\begin{equation}\label{stab:form:classic}
s_E^\partial(v,w)= \sum_{i=1}^{Nk} \Xi_i^\partial(v) \Xi_i^\partial(w) .
\end{equation}

We may call it {\em the identity matrix choice} since in the implementation procedure of the method, the bilinear form \eqref{stab:form:classic} is clearly associated with an identity matrix of dimension $Nk$.

% ------
% LEMMA TOLTO PERCHE' NON USATO
% ------
% Noting that any function of $V_E$ is continuous and piecewise polynomial on the boundary, an immediate adaptation of the results of Lemma 3.1 in [Bertoluzza, Math of Comp Vol. 73] gives the following result.
% \begin{lem}\label{lem:bert}
% Let A1 hold. For all $E \in \Omega_h$ and all $v_h \in V_E$ we have
% $$
% \| v_h \|_{L^{\infty}(\partial E)}^2 \lesssim C(h) \: ( h_E^{-1} \| v_h \|_{0,\partial E}^2 + |v_h|_{1/2,\partial E}^2) ,
% $$
% with $C(h) = (\log{(1+h_E/h_{m(E)})})^2$ and $h_{m(E)}$ the length of the smallest edge of $E$.
% \end{lem}
%

{
\begin{lem}\label{lem3}
Let assumptions A1 and A2 hold. For all $E \in \Omega_h$ and all $v_h \in V_E$ we have
\begin{equation}\label{eq:infty-est}
|v_h|_{1/2,\partial E}^2 \lesssim \widehat C(E) \| v_h \|_{L^{\infty}(\partial E)}^2
,
\end{equation}
with $\widehat C(E) = (\log{(1+h_E/h_{m(E)})})$.
\end{lem}

\begin{proof}
We first recall that $\partial E$ is meshed by means of its edges, so that $\partial E =\cup_{j=1}^N e_j$. We also define $h_j := |e_j|$. Moreover, in the proof we will make use of the space $H^{1/2}_{00}(\Gamma)$, where $\Gamma$ is a connected part of $\partial E$ with $|\Gamma|>0$. This space is defined by, see \cite{Lions-Magenes}:

\begin{equation}\label{h1200}
 H^{1/2}_{00}(\Gamma) = \Big\{  v\in H^{1/2}(\Gamma)\ :\ {\rm Ext}(v)\in H^{1/2}(\partial E)    \Big\} ,
\end{equation}
where ${\rm Ext}(v)$ denotes the extension by zero of $v$ to the whole $\partial E$.
Its norm

\begin{equation}\label{h1200-norm}
\displaystyle{ || v ||_{H^{1/2}_{00}(\Gamma)} := \left( |v |_{1/2, \Gamma}^2 +
 \int_\Gamma \frac{v(x)^2}{\rho(x)} \: {\rm d} x \right)^{1/2} } ,
\end{equation}
where $\rho(x)$ denotes the distance of $x$ from $\partial \Gamma$, is equivalent to
$ | {\rm Ext}(v) |_{1/2,\partial E}$.

Given $v_h\in V_E$, we set $v_L\in V_E$ as the usual piecewise linear Lagrange interpolant of $v_h$, relative to the edge mesh. We have

\begin{equation}\label{eq:infty-est-2}
|v_h|_{1/2,\partial E}^2 \lesssim |v_h - v_L|_{1/2,\partial E}^2 + |v_L|_{1/2,\partial E}^2.
\end{equation}

We now define $w_j = \chi_{e_j}(v_h-v_L)$ and we notice that, since $v_h -v_L$ vanishes at all the nodes, we have

\begin{equation}\label{eq:infty-est-3}
|v_h - v_L|_{1/2,\partial E} = \Big| \sum_{j=1}^N {\rm Ext}(w_j) \Big|_{1/2,\partial E}
\le \sum_{j=1}^N | {\rm Ext}(w_j) |_{1/2,\partial E}\lesssim \sum_{j=1}^N || w_j ||_{H^{1/2}_{00}(e_j)} .
\end{equation}
Exploiting that $w_j$ is a polynomial of degrees $\le k$ on $e_j$, a scaling argument shows that
$$
|| w_j ||_{H^{1/2}_{00}(e_j)}\lesssim || w_j ||_{L^{\infty}(e_j)}.
$$
Therefore, recalling assumption A2 and using that
$||v_L ||_{L^{\infty}(\partial E)} \lesssim ||v_h ||_{L^{\infty}(\partial E)} $, it holds
\begin{equation}\label{eq:infty-est-4}
|v_h - v_L|_{1/2,\partial E} \lesssim \sum_{j=1}^N || w_j ||_{L^{\infty}(e_j)} \lesssim || v_h - v_L ||_{L^{\infty}(\partial E)} \lesssim  ||v_h ||_{L^{\infty}(\partial E)} ,
\end{equation}
by which

\begin{equation}\label{eq:infty-est-4b}
|v_h - v_L|_{1/2,\partial E}^2  \lesssim  ||v_h ||_{L^{\infty}(\partial E)}^2 .
\end{equation}

It remains to estimate $|v_L|_{1/2,\partial E}^2$. We denote by $\varphi_i$ the usual hat function with support $\sigma_i := e_{i-1}\cup e_i$ (here $i-1$ and $i$ are intended modulo N). We write
$$
v_L = \sum_{i=1}^N v_i\varphi_i ,
$$
where $v_i\in\mathbb{R}$ is the value of $v_L$ at the $i$-th node. We have, using assumption A2:

\begin{equation}\label{eq:infty-est-5}
|v_L|_{1/2,\partial E}^2
\lesssim || v_L||_{L^\infty(\partial E)}^2\sum_{i=1}^N | \varphi_i |_{1/2,\partial E}^2
\lesssim || v_L||_{L^\infty(\partial E)}^2\sum_{i=1}^N || \varphi_i ||_{H^{1/2}_{00}(\sigma_i)}^2.
\end{equation}

Recalling \eqref{h1200-norm}, direct computations show that

\begin{equation}\label{norm_i}
 | \varphi_i |_{H^{1/2}(\sigma_i)}^2\lesssim 1\qquad ; \qquad
 \int_{\sigma_i} \frac{\varphi_i(x)^2}{\rho(x)}{\rm d}\, x
 \lesssim  \log \left(1 + \frac{ \max \{ h_{i-1},h_i \} }{ \min \{ h_{i-1},h_i \} } \right) ,
\end{equation}
by which we obtain

\begin{equation}\label{eq:infty-est-6}
|| \varphi_i ||_{H^{1/2}_{00}(\sigma_i)}^2\lesssim
 \log \left(1 + \frac{ \max \{ h_{i-1},h_i \} }{ \min \{ h_{i-1},h_i \} } \right) ,
\end{equation}

Therefore, using again assumption A2 and noting that

$$
\frac{ \max \{ h_{i-1},h_i \} }{ \min \{ h_{i-1},h_i \} }
\le \frac{ h_E }{ h_{m(E)} } \qquad 1\le i\le N ,
$$
from \eqref{eq:infty-est-5} and \eqref{norm_i} we get

\begin{equation}\label{eq:infty-est-7}
|v_L|_{1/2,\partial E}^2 \lesssim
 \log  \left(1 + \frac{ h_E }{ h_{m(E)} } \right) || v_L||_{L^\infty(\partial E)}^2.
\end{equation}

Combining \eqref{eq:infty-est-2}, \eqref{eq:infty-est-4b} and \eqref{eq:infty-est-7}, we get \eqref{eq:infty-est}.

\end{proof}
}

We now have the following stability result.

\begin{thm}\label{thm-stab}
Let assumptions A1 and A2 hold. Then, for the boundary form \eqref{stab:form:classic},
and for any of the choices \eqref{eq:Rint}, \eqref{eq:Rbnd}, \eqref{eq:Rvarg} of the operator $\PR$, conditions \eqref{eq:hyp:1} and \eqref{eq:hyp:2} hold with positive constants $\widehat C_1$ and $\widehat C_2$ that satisfy
\begin{equation}\label{ident-stab}
\widehat C_1(E) \lesssim (\log{(1+h_E/h_{m(E)})})  \ , \quad \widehat C_2(E) \lesssim 1 .
\end{equation}

\end{thm}
\begin{proof}

Standard results for polynomials in one dimension immediately give
\begin{equation}\label{stoinf}
\| w_h \|_{L^{\infty}(\partial E)}^2 \lesssim s_E^\partial(w_h,w_h) \qquad \forall w_h\in V_E .
\end{equation}
%
%Therefore, using H\"older's inequality, we get
%\begin{equation}\label{eq:L2-est}
%\begin{aligned}
%h_E^{-1} \| v_h -\PR v_h\|_{L^2(\partial E)}^2 & \lesssim \| v_h -\PR v_h\|_{L^\infty(\partial E)}^2\\
%& \lesssim s_E^\partial\left( (I -\PR) v_h, (I -\PR) v_h\right) \qquad \forall v_h\in V_E.
%\end{aligned}
%\end{equation}
%
A combination of \eqref{stoinf} and Lemma \ref{lem3} yields:
\begin{equation}\label{eq:half-est}
\begin{aligned}
|v_h|_{1/2,\partial E}^2 & = |v_h - \PR v_h|_{1/2,\partial E}^2
\lesssim \widehat C(E)\, s_E^\partial\left( (I -\PR) v_h, (I -\PR) v_h\right) \\
&\lesssim \widehat C(E)\, \left( s_E^\partial\left( (I -\PR) v_h, (I -\PR) v_h\right)
+ |\Pi_E v_h |_{1,E}^2)\right)
\qquad \forall v_h\in V_E ,
\end{aligned}
\end{equation}
i.e. condition \eqref{eq:hyp:1} holds with $\widehat C_1(E) \lesssim (\log{(1+h_E/h_{m(E)})})$.

We now prove that estimate \eqref{eq:hyp:2} holds.
%
%
% %%%%%%%%%%%%%%%%%%%%%%%%%%%%%%%%%%
%
%
% % Combining Theorem \ref{th:main} and \ref{thm-stab} one obtains, under A1 and A2, that
% % \begin{equation}\label{sumbound}
% % (\log{(1+h_E/h_e)})^{-2} \: a_E(v,v)  \lesssim  s_E(v,v) \lesssim a_E(v,v)
% % \quad \forall v \in V_E \textrm{ with } \Pi v = 0 .
% % \end{equation}
% % Following the simple arguments in [VOLLEY+++], page XXX, from Lemma \ref{lem:A2} and bound \eqref{sumbound} we get the stability property
% % \begin{equation}\label{fin-bound}
% % (\log{(1+h_E/h_e)})^{-2} \: a_E(v,v)  \lesssim  a_E^h(v,v) \lesssim a_E(v,v) \quad \forall v \in V_E .
% % \end{equation}
% % where the above form was defined in \eqref{ah-form}.
%
%
Recalling assumption A2, it is immediate to check that
\begin{equation}\label{eq:id-infty}
s_E^\partial(v,v) \lesssim N \| v \|_{L^{\infty}(\partial E)}^2 \lesssim \| v \|_{L^{\infty}(\partial E)}^2 \qquad\forall v\in C^0(\partial E).
\end{equation}
Take any $p \in \Pp_k(E)$. We get, using bound \eqref{eq:id-infty} , an inverse estimate for polynomials (cf. Remark \ref{rm:vem-invest}), and recalling either \eqref{R-approx} or \eqref{R-approx-pol} (depending on the choice of the operator $\PR$):

\begin{equation}\label{eq:lowercond-est}
\begin{aligned}
s_E^\partial\big( (I-\PR) p , (I-\PR) p \big) &\lesssim
\| (I-\PR) p  \|_{L^{\infty}(\partial E)}^2 \lesssim \| (I-\PR) p  \|_{L^{\infty}(E)}^2 \\
& \lesssim  h_E^{-2} || (I-\PR) p||_{0,E}^2
\leq | p  |_{1,E}^2 ,
\end{aligned}
\end{equation}
i.e. condition \eqref{eq:hyp:2} holds with $\widehat C_2(E) \lesssim 1$.
\end{proof}

The following corollary shows that, even in the presence of arbitrarily small edges (provided the number of edges are uniformly bounded), the convergence rate of the Virtual Element Method is quasi-optimal, in the sense that only a logarithmic factor is lost.

\begin{cor}\label{cor:ident}
Let assumptions A1 and A2 hold. Let $u$ be the solution of problem \eqref{cont-pbl}, assumed to be in $H^s(\Omega)$, $s>1$. Let $u_h$ be the solution of the discrete problem \eqref{discr-pbl}. Then it holds
$$
\| u - u_h \|_{1,\Omega} \lesssim c(h) \: h^{s-1} |u|_{s,\Omega} \qquad 1 < s \le k+1 .
$$
with
$$
c(h) = \max_{E \in \Omega_h} \left(\log{(1+h_E/h_{m(E)})}\right) .
$$
If the stronger assumption A3 holds, then clearly $c(h)\lesssim 1$.
\end{cor}
\begin{proof}
% The proof easily follows by noting that:
%
% \begin{enumerate}
%  %
%  \item Theorem \ref{thm-stab} allows to apply Propositions \ref{prop:upperbound}
%  and \ref{prop:lowerbound}. Therefore, assumptions \eqref{ass-1} and \eqref{ass-2} hold
%  with $C_1\lesssim (\log{(1+h_E/h_{m(E)})})^2$ and $C_2\lesssim 1$, respectively.
%  %
%  \item
%  %
% \end{enumerate}

% The proof follows from the stability result \eqref{fin-bound} and Theorem XXX in [Volley+++], combined with the approximation results for virtual interpolants in Theorem \ref{thm:approx} and standard polynomial approximation on polygons [Scott-Dupont+++].

Theorem \ref{thm-stab} allows to apply Propositions \ref{prop:upperbound}
 and \ref{prop:lowerbound}. Therefore, assumptions \eqref{ass-1} and \eqref{ass-2} hold
 with $C_1(E)\lesssim \log{(1+h_E/h_{m(E)})}$ and $C_2(E)\lesssim 1$, respectively.
Then, Theorem \ref{teoconv} can be invoked; a look at the constants shows that
$$
C_{err}(h) \lesssim 1 + \max_{E \in \Omega_h} \{ C_1(E) \} \lesssim c(h) .
$$
We now estimate the terms in the right-hand side of \eqref{teoconv:eq}.
We first recall \eqref{defrhs-ter}:
\begin{equation}\label{eq:id-00}
\mathfrak{F}_h \lesssim h^k .
\end{equation}

Moreover, Theorem \ref{thm:approx} shows that

\begin{equation}\label{eq:id-3}
 | u-u_I |_{1,\Omega}
 \lesssim \Big(\sum_{E\in\Th} h_E^{2s-2}| u |_{s,E}^2\Big)^{1/2}
 \lesssim h^{s-1} | u |_s \qquad 1 < s\le k+1 .
\end{equation}
while standard approximation results on polygons (see for instance \cite{scott-dupont}) yield
\begin{equation}\label{eq:id-4}
 \big( \sum_{\E\in\Th} |u-\up|_{1,\E}^2\big)^{1/2}
 \lesssim \Big(\sum_{E\in\Th} h_E^{2s-2}| u |_{s,E}^2\Big)^{1/2}
 \lesssim h^{s-1} | u |_s  \qquad 1 < s\le k+1 .
\end{equation}
We now look into the term $\tri u-u_I \tri$.
From \eqref{st-term}, \eqref{st-seminorm} and \eqref{st-norm}, we deduce that we need to estimate:

\begin{enumerate}

\item the term

\begin{equation}\label{eq:pesti}
 a_E\big(\Pi_E(u-u_I), \Pi_E(u-u_I) \big) \ ;
\end{equation}

\item the term
\begin{equation}\label{eq:id-5}
\begin{aligned}
 s_E^\partial((I-\PR)(u-u_I),(I-\PR)(u-u_I))
 =  \sum_{i=1}^{Nk} \Xi_i^\partial((I-\PR)(u-u_I))^2 \ ;
\end{aligned}
\end{equation}

\item the term, see \eqref{stab-int-bound},

\begin{equation}\label{eq:id-6}
 s_E^\circ((I-\PR)(u-u_I),(I-\PR)(u-u_I))
 =  \sum_{i=1}^{k(k-1)/2} \Xi_i^\circ((I-\PR)(u-u_I))^2\ .
\end{equation}
\end{enumerate}

Clearly, if choice \eqref{no-int-stab} is used instead of \eqref{stab-int-bound}, this last term vanishes.
Take $s$ with $1<s<k+1$, and $\varepsilon$ such that $0<\varepsilon < \min\{1/2, s-1 \}$.
Regarding \eqref{eq:pesti}, we notice that from the continuity of $\Pi_E$ and from Theorem \ref{thm:approx} with $\sigma=1$, it holds

\begin{equation}\label{eq:pesti2}
 a_E\big(\Pi_E(u-u_I), \Pi_E(u-u_I) \big) \lesssim |u-u_I|_{1,E}^2\lesssim  h^{2s-2}_E|u|_{s,E}^2 .
\end{equation}

Now, the Sobolev embedding $H^{1/2+\varepsilon}(\partial E) \subset C^0(\partial E)$ shows that it holds:

\begin{equation}\label{eq:sob-emb}
 || v ||_{L^\infty(\partial E)} \lesssim  h_E^{-1/2} ||v||_{0,\partial E}
+  h_E^{\varepsilon} |v|_{1/2+\varepsilon,\partial E}   \qquad \forall v\in H^s(E).
\end{equation}
A scaled trace inequality, that can be derived by an argument analogous to that in Lemma \ref{lem:tr:um}, gives
\begin{equation}\label{eq:tr-in}
h_E^{-1/2} ||v||_{0,\partial E}  +  h_E^{\varepsilon} |v|_{1/2+\varepsilon,\partial E}
\lesssim   h_E^{-1} ||v||_{0,E}  +  h_E^{\varepsilon} |v|_{1+\varepsilon, E}  \qquad \forall v\in H^s(E) .
\end{equation}
Therefore, \eqref{eq:id-infty}, \eqref{eq:sob-emb} and \eqref{eq:tr-in} yield

\begin{equation}\label{eq:int-id-1}
s_E^\partial(v,v)
\lesssim   h_E^{-2} ||v||_{0,E}^2  +  h_E^{2\varepsilon} |v|_{1+\varepsilon, E}^2  \qquad \forall v\in H^s(E).
\end{equation}
Choosing $v=(I-\PR)(u-u_I)_{|E}$ in \eqref{eq:int-id-1}, using \eqref{R-approx} or \eqref{R-approxh1} (depending on the choice of the operator $\PR$) and noting
that $ |(I-\PR)(u-u_I)|_{1+\varepsilon, E}= |u-u_I|_{1+\varepsilon, E}$, we obtain:
\begin{equation}\label{eq:int-id-2}
s_E^\partial((I-\PR)(u-u_I),(I-\PR)(u-u_I))
\lesssim  |u-u_I|_{1, E}^2 + h_E^{2\varepsilon} |u-u_I|_{1+\varepsilon, E}^2 .
\end{equation}
An application of Theorem \ref{thm:approx} with $\sigma=1$ (resp., $\sigma=1+\varepsilon$) in the first (resp., second) term of the right-hand side of \eqref{eq:int-id-2} leads to:

\begin{equation}\label{eq:int-id-3}
s_E^\partial((I-\PR)(u-u_I),(I-\PR)(u-u_I))
\lesssim   h^{2s-2}_E|u|_{s,E}^2   .
\end{equation}
We now notice that bound \eqref{X-newnew} applies also to $(u-u_I)_{|E}$, and not only to polynomials $p\in \Pp_k$.
Therefore, by using again \eqref{R-approx} or \eqref{R-approxh1} (depending on the choice of $\PR$) and Theorem \ref{thm:approx} one easily gets
\begin{equation}\label{eq:id-9}
 s_E^\circ((I-\PR)(u-u_I),(I-\PR)(u-u_I))
\lesssim h_E^{-2} \| (I-\PR)(u-u_I) \|_{0,E}^2
\lesssim |u-u_I|_{1,E}^2 \lesssim h^{2s-2}_E |u|_{s,E}^2 .
\end{equation}
Combining \eqref{eq:pesti2}, \eqref{eq:int-id-3} and \eqref{eq:id-9}, we get
\begin{equation}\label{eq:id-10}
 \tri u-u_I \tri
 \lesssim \Big( \sum_{E\in \Th} h^{2s-2}_E |u|_{s,E}^2 \Big)^{1/2} \lesssim h^{s-1} |u|_{s}
 \qquad 1 < s\le k+1 .
\end{equation}

By following the same steps and using standard approximation results on polygons (see for instance \cite{scott-dupont}), we get

\begin{equation}\label{eq:id-11}
 \tri u-u_\pi \tri
 \lesssim \Big( \sum_{E\in \Th} h^{2s-2}_E |u|_{s,E}^2 \Big)^{1/2} \lesssim h^{s-1} |u|_{s}
 \qquad 1 < s\le k+1 .
\end{equation}

We conclude by collecting estimates \eqref{eq:id-00}, \eqref{eq:id-10}, \eqref{eq:id-11},
\eqref{eq:id-3} and \eqref{eq:id-4}.
\end{proof}

% ---------------------------------------------------------------------------
\subsubsection{A ``classical'' stability bound}\label{ssec:memo}
% ---------------------------------------------------------------------------
%%%%%
We close this part on the identity matrix choice by showing that the classical stability result of \cite{volley}, see equation (3.7) of \cite{volley}, can also be proved under the more general mesh assumptions considered in this paper. This result could be used to prove the same error estimate as in Corollary \ref{cor:ident} without resorting to the approach described in this paper, but simply applying the standard theory of \cite{volley}.
We need an additional preliminary lemma.
\begin{lem}\label{lem:ani-trace}
Let assumption A1 hold. We have
\begin{equation}\label{B1}
h_E^{-1} \| v \|_{0,\partial E}^2 \lesssim h_E^{-2} \| v \|_{0,E}^2 + |v|_{1,E}^2  \qquad  \forall E \in \Omega_h ,
\end{equation}
and for all $v$ in $H^1(E)$.
\end{lem}
\begin{proof}
The simple proof is based on an anisotropic scaling argument. Take an edge $e \in \partial E$, and let $T \in {\cal T}_h$ be the associated triangle (see the proof of Theorem \ref{thm:approx}). By a rotation and translation of the cartesian $(x,y)$-coordinates, it is not restrictive to assume that $e = \{0\} \times [- h_e/2 , h_e/2]$, and that the center of the ball, see assumption A1, ${\bf x}_E = (x_E, y_E)$ satisfies $x_E \ge 0$. As a consequence of assumption A1, it is easy to check that $x_E \lesssim h_E$, $y_E \lesssim h_E$ and that the ball $B_E$ is contained in the half plane $\{ (x,y) \in {\mathbb R}^2 : x \ge 0 \}$.
Therefore, we also have $x_E \gtrsim h_E$.
Let now $\hat T$ be the triangle of vertexes $(0,h_e/2), \: (0, -h_e/2), \: (h_e/2,0)$. We now consider the unique affine mapping $F : T \rightarrow \hat T$ that leaves the edge $e$ (and its orientation) unchanged: $\hat e := F(e)=e$.
By an explicit computation of $F$ and its inverse $F^{-1}$, we get the Jacobian matrices
$$
DF = \begin{pmatrix} h_e/x_E  & 0 \\ - y_E/x_E & 1 \end{pmatrix} \ , \quad
DF^{-1} = \begin{pmatrix} x_E/h_e  & 0 \\ y_E/h_e & 1 \end{pmatrix} .
$$
The proof now follows by a scaling argument. Indeed, denoting $\hat v = v \circ F^{-1}$, well known (scaled) trace estimates on $\hat T$ and a simple change of variables give
$$
\begin{aligned}
& \| v \|_{0,e}^2 = \| \hat v \|_{0,\hat e}^2 \lesssim h_e^{-1} \| \hat v \|_{0,\hat T}^2 + h_e | \hat v |_{1,\hat T}^2 \\
& \lesssim h_e^{-1} \frac{h_e}{x_E} \| v \|_{0,T}^2
+  h_e \frac{h_e}{x_E} \left( (\frac{x_E}{h_e})^2 + (\frac{y_E}{h_e})^2 \right) \| \frac{\partial v}{\partial x} \|_{0,T}^2
+  h_e \frac{h_e}{x_E} \| \frac{\partial v}{\partial y} \|_{0,T}^2 .
\end{aligned}
$$
By recalling the upper and lower bounds on $(x_E,y_E)$, the above estimate yields
$$
\| v \|_{0,e}^2 \lesssim h_E^{-1} \| v \|_{0,T}^2 + h_E \| \nabla v \|_{0,T}^2 ,
$$
that immediately implies \eqref{B1} by summing over all $e \in \partial E$.
\end{proof}

{
% ---
\begin{prop}\label{prop:class-stab}
Let assumptions A1 and A2 hold. Then, for any of the choices \eqref{eq:Rint}, \eqref{eq:Rbnd}, \eqref{eq:Rvarg} of the operator $\PR$, it holds
\begin{equation}\label{boh}
s_E(v_h,v_h) \lesssim a_E(v_h,v_h) \lesssim c(h) \: s_E(v_h,v_h) \quad \forall v_h \textrm{ with } \Pi_E v_h = 0 ,
\end{equation}
where
$$
c(h) = \max_{E \in \Omega_h} \log{(1+h_E/h_{m(E)})} .
$$
Note that if the stronger assumption A3 holds, then clearly $c(h)\lesssim 1$.
\end{prop}
\begin{proof}
Theorem \ref{thm-stab} combined with Proposition \ref{prop:upperbound} gives the validity of \eqref{ass-1} with constant $C_1(E) \lesssim c(h)$. Since $\Pi_E v_h = 0$ implies $\PR v_h=0$, bound \eqref{ass-1} yields the second inequality in \eqref{boh}.

We now show the other bound. We first consider $s_E^\circ( v_h , v_h )$ and notice that there is nothing to estimate if the choice \eqref{no-int-stab} has been done. Therefore, we focus on the choice \eqref{stab-int-bound}.
By definition, the Holder inequality and recalling $\| m_i \|_{L^\infty} \lesssim 1$, we get the estimate:
%%
% Note that the arguments in the proof of Proposition \ref{prop:lowerbound} can be applied identically also in the case of a general function $v_h \in \VE$, there is no need to have a polynomial. Therefore we obtain
%%
\begin{equation}\label{st:prop:int}
s_E^\circ\big( v_h , v_h \big) = \sum_i^{n_{k-2}} \Xi_i^\circ\big( v_h \big)^2 = \sum_i^{n_{k-2}}  |E|^{-2} \Big( \int_E   v_h \: \ m_i \Big)^2
\lesssim \sum_i^{n_{k-2}}  |E|^{-1} \|  v_h \|_{L^2(E)}^2 ,
\end{equation}
for all $v_h \in \VE$.

Regarding the term $s_E^\partial( v_h , v_h )$, due to assumption A2, it is immediate to derive
$$
s_E^\partial\big( v_h , v_h \big) \lesssim \| v_h \|_{L^\infty(\partial E)}^2
$$
that, using Lemmas \ref{lem:A2}, \ref{lem:tr:um} and \ref{lem:ani-trace} yields
\begin{equation}\label{XX-bnd}
s_E^\partial\big( v_h , v_h \big) \lesssim h_E^{-2} \|  v_h \|_{L^2(E)}^2
+ | v_h |_{H^1( E)}^2.
\end{equation}
Since $\Pi v_h = 0$ implies $\PR v_h=0$, either bound \eqref{R-approx} (for the choices \eqref{eq:Rint} and \eqref{eq:Rbnd}), or bound \eqref{R-approxh2} (for the choice \eqref{eq:Rvarg}), yields
\begin{equation}\label{XX-eq}
h_E^{-2} \|  v_h \|_{L^2(E)}^2 \lesssim | v_h |_{H^1( E)}^2 .
\end{equation}
The first bound in \eqref{boh} now follows from combining \eqref{st:prop:int}, \eqref{XX-bnd} and \eqref{XX-eq} and noting that $|v_h|_{1,E}^2\lesssim a_E(v_h,v_h)$.
\end{proof}
}

% ----------------------------------------------------------------------------------------------
\subsection{A stabilization based on boundary derivatives}
\label{ssec:3:2}
% ----------------------------------------------------------------------------------------------

We now analyze a different choice for the boundary part of the stabilization term, namely the one given by (cf. \cite{wriggers}):
\begin{equation}\label{wriggers_ch}
s_E^\partial(v_h,w_h) = h_E \int_{\partial E} \partial_s v_h \: \partial_s w_h \: {\rm d}s
\qquad \forall v_h, w_h \in V_E .
\end{equation}
We highlight that, contrary to the identity matrix stabilization presented in section \ref{ssec:3:1}, the standard approach of \cite{volley} applied to \eqref{wriggers_ch}, would lead to a strongly suboptimal result in the presence of small edges.
Indeed, the term \eqref{wriggers_ch} can be bounded by the $H^1$ semi-norm only with a constant $\alpha^* \simeq h_E/h_{m(E)}$ (cf. the second bound in equation (3.7) of \cite{volley}).
{\em Instead, with the present analysis we can obtain uniform bounds only making use of assumption A1.}
In fact, we have the following result.

{
\begin{thm}\label{thm-stab-wrg}
Let assumption A1 hold. Then, for the boundary form \eqref{wriggers_ch}, and for any of the choices \eqref{eq:Rint}, \eqref{eq:Rbnd} and \eqref{eq:Rvarg} of the operator $\PR$, conditions \eqref{eq:hyp:1} and \eqref{eq:hyp:2} hold with positive constants $\widehat C_1$ and $\widehat C_2$ that satisfy
\begin{equation}\label{ident-stab-wrg}
\widehat C_1(E) \lesssim 1  \ , \quad \widehat C_2(E) \lesssim 1 .
\end{equation}

\end{thm}
\begin{proof}

We first prove that condition \eqref{eq:hyp:2} is fulfilled.
Take any $p \in \Pp_k(E)$. Using assumption A1 and an inverse inequality for polynomials (cf. Remark \ref{rm:vem-invest}), we get

\begin{equation}\label{eq:lowercond-est-wrg}
\begin{aligned}
s_E^\partial\big( (I-\PR) p , (I-\PR) p \big) &
= h_E | (I-\PR) p  |_{1,\partial E}^2
\lesssim h_E^2 \| \nabla (I-\PR) p  \|_{L^\infty(\partial E)}^2 \\
& \le h_E^2 \| \nabla (I-\PR) p  \|_{L^\infty(E)}^2
\lesssim \| \nabla (I-\PR) p  \|_{0, E}^2
= | p  |_{1,E}^2 ,
\end{aligned}
\end{equation}
i.e. condition \eqref{eq:hyp:2} holds with $\widehat C_2(E) \lesssim 1$.

To prove that condition \eqref{eq:hyp:1} is fulfilled,
we simply notice that
\begin{equation}\label{eq:wrg-1}
 | v_h |_{1/2,\partial E}^2 \lesssim h_E |v_h|_{1,\partial E}^2
 = h_E |(I-\PR)v_h|_{1,\partial E}^2 = s_E^\partial((I-\PR)v_h,(I-\PR)v_h) \qquad \forall v_h\in V_E .
\end{equation}
and we obtain that condition \eqref{eq:hyp:1} holds with $\widehat C_1(E) \lesssim 1$.
\end{proof}
}

\begin{cor}\label{cor:wrg}
Let assumption A1 hold. Let $u$ be the solution of problem \eqref{cont-pbl}, assumed to be in $H^s(\Omega)$, $s>3/2$. Let $u_h$ be the solution of the discrete problem \eqref{discr-pbl}, with the choice \eqref{wriggers_ch}. Then it holds
$$
\| u - u_h \|_{1,\Omega} \lesssim  \: h^{s-1} |u|_{s,\Omega} \qquad 3/2 < s \le k+1 .
$$

\end{cor}
\begin{proof}
Theorem \ref{thm-stab-wrg} allows to apply Propositions \ref{prop:upperbound}
 and \ref{prop:lowerbound}. Therefore, assumptions \eqref{ass-1} and \eqref{ass-2} hold
 with $C_1(E)\lesssim 1$ and $C_2(E)\lesssim 1$, respectively.
Then, Theorem \ref{teoconv} can be invoked with $C_{err}(h)$ satisfying
$C_{err}(h)\lesssim 1$. We now estimate the terms in the right-hand side of \eqref{teoconv:eq}.
Using exactly the same arguments of Corollary \ref{cor:ident}, and focusing again only on the non-trivial choice \eqref{stab-int-bound}, we get:
\begin{equation}\label{eq:wrg-00}
\mathfrak{F}_h \lesssim h^k ,
\end{equation}
\begin{equation}\label{eq:wrg-3}
\big( \sum_{\E\in\Th}  |\Pi_E(u-u_I)|_{1,\E}^2\big)^{1/2}
\lesssim| u-u_I |_{1,\Omega}
 \lesssim h^{s-1} | u |_s \qquad 3/2 < s\le k+1 ,
\end{equation}
\begin{equation}\label{eq:wrg-4}
\big( \sum_{\E\in\Th}  |\Pi_E(u-\up)|_{1,E}^2\big)^{1/2}
  \lesssim \big( \sum_{\E\in\Th} |u-\up|_{1,E}^2\big)^{1/2}
 \lesssim h^{s-1} | u |_s \qquad 3/2 < s\le k+1 ,
\end{equation}
and
\begin{equation}\label{eq:wrg-9}
 s_E^\circ((I-\PR)(u-u_I),(I-\PR)(u-u_I))
 \lesssim |u-u_I|_{1,E}^2 \lesssim h^{2s-2}_E |u|_{s,E}^2 \qquad 3/2 < s\le k+1 .
\end{equation}
Therefore, we only need to estimate the boundary part
$$
s_E^\partial((I- \PR)(u-u_I),(I- \PR)(u-u_I)) = s_E^\partial(u-u_I   ,u-u_I) .
$$
To this end, take $s>3/2$ and $\sigma$ such that $3/2< \sigma < s$.
We have, using a scaled trace inequality (use a similar argument to that in Lemma \ref{lem:tr:um} for the function $\nabla(u-u_I)$),

\begin{equation}\label{eq:wrg-7}
s_E^\partial(u-u_I ,u-u_I)  = h_E |u-u_I |^2_{1,\partial E}
 \lesssim   |u-u_I|_{1,E}^2 + h^{2\sigma - 2}_E |u-u_I|_{\sigma,E}^2.
\end{equation}
Hence, Theorem \ref{thm:approx} gives:

\begin{equation}\label{eq:wrg-8}
s_E^\partial(u-u_I  ,u-u_I)
\lesssim   |u-u_I|_{1,E}^2 + h^{2\sigma-2}_E |u-u_I|_{\sigma,E}^2
\lesssim  h^{2s-2}_E |u|_{s,E}^2  .
\end{equation}

Combining \eqref{eq:wrg-9} and \eqref{eq:wrg-8}, we get

\begin{equation}\label{eq:wrg-10}
 \tri u-u_I \tri
 \lesssim \Big( \sum_{E\in \Th} h^{2s-2}_E |u|_{s,E}^2 \Big)^{1/2} \lesssim h^{s-1} |u|_{s}.
\end{equation}

Similarly, using also standard approximation results on polygons (see \cite{scott-dupont}), we get

\begin{equation}\label{eq:wrg-11}
 \tri u-u_\pi \tri
 \lesssim \Big( \sum_{E\in \Th} h^{2s-2}_E |u|_{s,E}^2 \Big)^{1/2} \lesssim h^{s-1} |u|_{s}.
\end{equation}

We conclude by collecting estimates \eqref{eq:wrg-00}, \eqref{eq:wrg-3}, \eqref{eq:wrg-4},
\eqref{eq:wrg-10} and \eqref{eq:wrg-11}.

\end{proof}

\begin{remark}
The same analysis can be employed to prove error estimates for many other choices of the stabilization. We here spend some word on the following variants of choice \eqref{wriggers_ch}.

\smallskip\noindent
The first variant is an ``$L^2$-version'' of \eqref{wriggers_ch}:
\begin{equation}\label{wriggers_l2}
s_E^\partial(v_h,w_h) = \sum_{e\in\partial E} h_e^{-1} \int_{e} v_h \:  w_h {\rm d}\,s
\qquad \forall v_h, w_h \in V_E .
\end{equation}
Since it is easy to check that, under assumptions A1 and A2, it holds
$$
\sum_{e\in\partial E} h_e^{-1} \| v_h \|_{0,E}^2 \simeq \| v_h \|_{L^\infty}^2 \qquad \forall v_h \in V_{h|E},
$$
by following the same steps used for the identity matrix choice of Section \ref{ssec:3:1}, one can easily obtain that Theorem \ref{thm-stab} holds also for the present choice.

A second possible choice would be to substitute $h_e^{-1}$ with $h_E^{-1}$ in \eqref{wriggers_l2}. For this choice, robust results (at least from the theoretical perspective of the present analysis) would be obtained only under the stronger assumption A3. Indeed, it is easy to check that for this latter choice one would get a factor $h_E/h_{m(E)}$ in the constant of bound \eqref{eq:hyp:1}.

%\item A ``discrete $L^2$-version'' of \eqref{wriggers_ch}:
%
%\begin{equation}\label{wriggers_l2-d}
%s_E^\partial(v_h,w_h) =
%h_E \sum_{e\in\partial E} h_e^{-1} \Big( v_h(p_e^+) - v_h(p_e^-)\Big) \Big( w_h(p_e^+) - w_h(p_e^-)\Big)
%\qquad \forall v_h, w_h \in V_E ,
%\end{equation}
%%
%where, given an orientation of $\partial E$ (counterclockwise, for instance), $p_e^-$ and $p_e^+$ denotes the first and the second endpoint of $e$, respectively. We also notice that for $k=1$, the choice \eqref{wriggers_l2-d} is indeed identical to \eqref{wriggers_ch}.
\end{remark}

% --------------------------------------------------------------------------------
\section{Numerical tests}\label{sec:4}
% --------------------------------------------------------------------------------

For all numerical tests we will consider Laplace equation on the
unit square $\Omega:=]0,1[^2$:
\begin{equation}
\left\{
\begin{aligned}
-\Delta u &= f \quad\text{in }\Omega \\
        u &= g \quad\text{on }\partial\Omega
\end{aligned}
\right.
\end{equation}
with right hand side $f$ and Dirichlet boundary condition $g$
defined in such a way that the exact solution is
\begin{equation}
\uex(x,y) := x^3 - xy^2 + x^2y + x^2 -xy -x + y -1 +
\sin(5x)\sin(7y) + \log(1+x^2+y^4).
\end{equation}
In the following brief experiments we will address some of the
issues considered in the paper.
\begin{figure}[ht]
\centerline{
\includegraphics[width=0.7\textwidth]{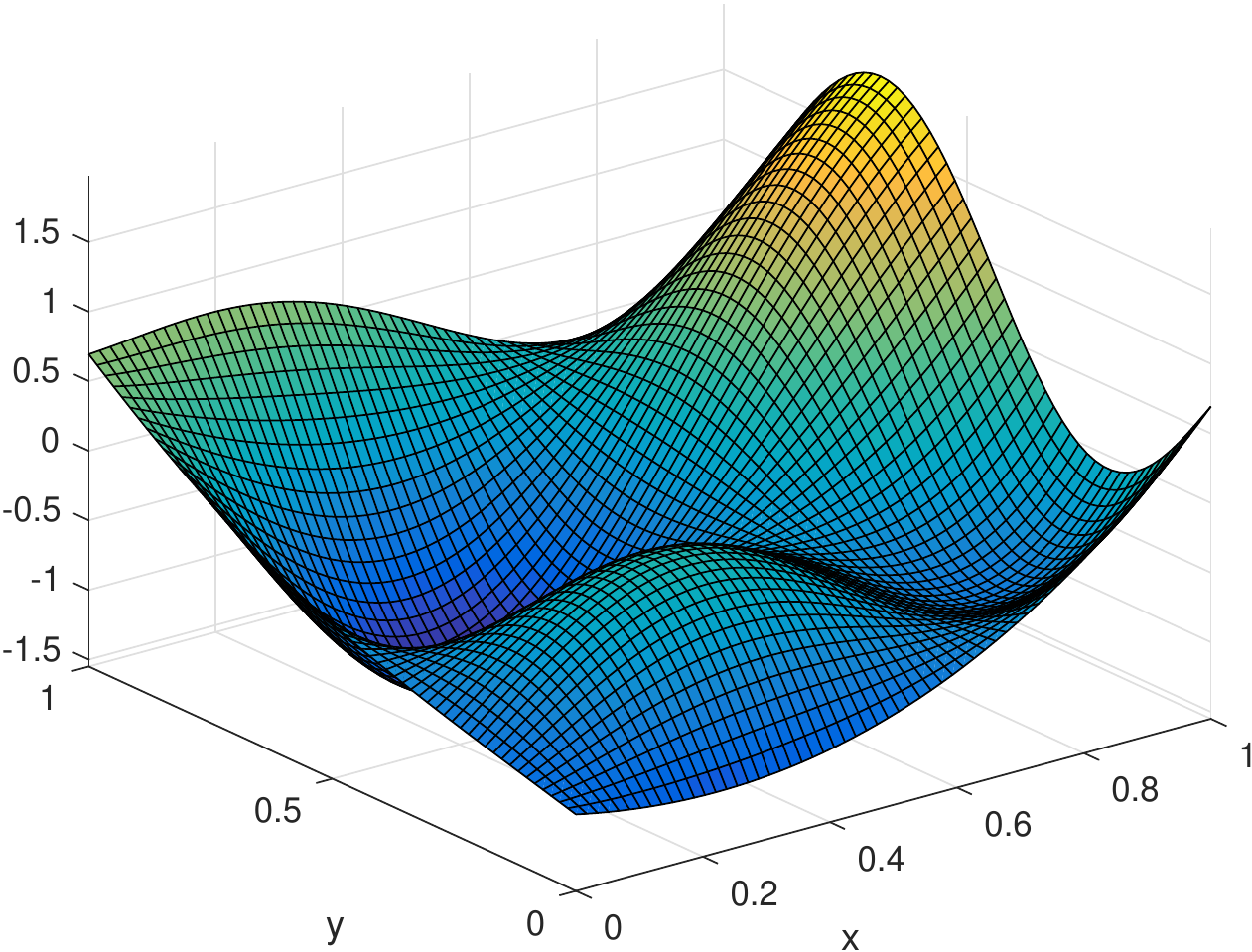}
} \caption{Exact solution} \label{exact-solution}
\end{figure}

\begin{remark}
First of all, we point out that in all our experiments we have
observed a very weak dependence of the VEM solution with respect
to the inclusion in the stabilization of the term $s_E^\circ(v,w)$
depending on the internal degrees of freedom. %%
This observation holds for all kinds of boundary stabilization
adopted. %%
Hence, we have set everywhere $s_E^\circ(v,w)=0$ (see Subsection
\ref{ssec:boundaryred}).
\end{remark}

% ----------
\subsection{Small edges}
In the first numerical experiment we consider the issue of the
presence of very small edges. On the one hand, we show that the
classical VEM stabilization \eqref{stab:form:classic} can generate
small oscillations, that are of the order of the approximation
error. In the case $k=1$ these oscillations are visible and,
depending on the application, may be preferable to avoid. However,
already for $k=2$ the oscillations become so small to be
practically negligible. %%
On the other hand, we show that the stabilization
\eqref{wriggers_ch} eliminates this oscillations already for the
$k=1$ case.

We consider a mesh obtained by gluing together two distinct meshes
along $x=0.5$; this case can happen for instance in contact
problems, see \cite{wriggers}. The mesh is shown in
Fig.\,\ref{fig:mesh-joined}, while in Fig.\,\ref{fig:mesh-section}
we show the section of the mesh at $x=0.5$. Note that around
$y=0.4$ there is a very small edge of length $3.21\times10^{-4}$.
\begin{figure}[ht]
\hfill
 \begin{minipage}[b]{0.45\textwidth}
  \centerline{
  \includegraphics[width=\textwidth]{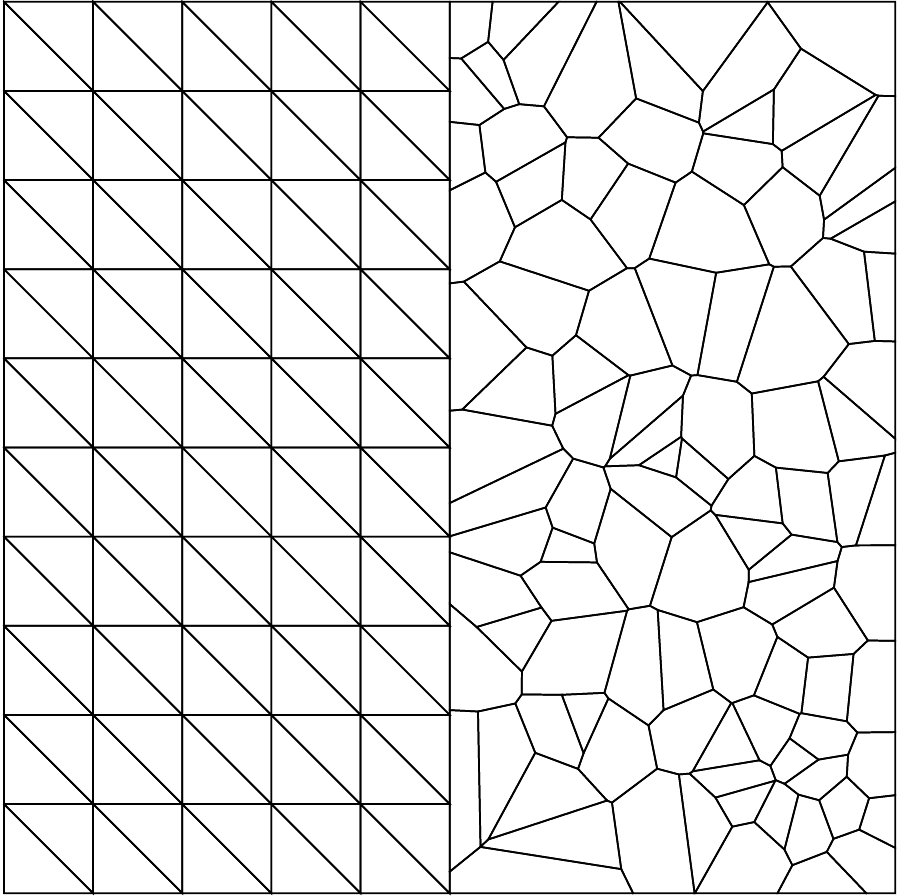}
  }
  \caption{Mesh}
 \label{fig:mesh-joined}
 \end{minipage}
\hfill
 \begin{minipage}[b]{0.45\textwidth}
  \centerline{
  \includegraphics[width=\textwidth]{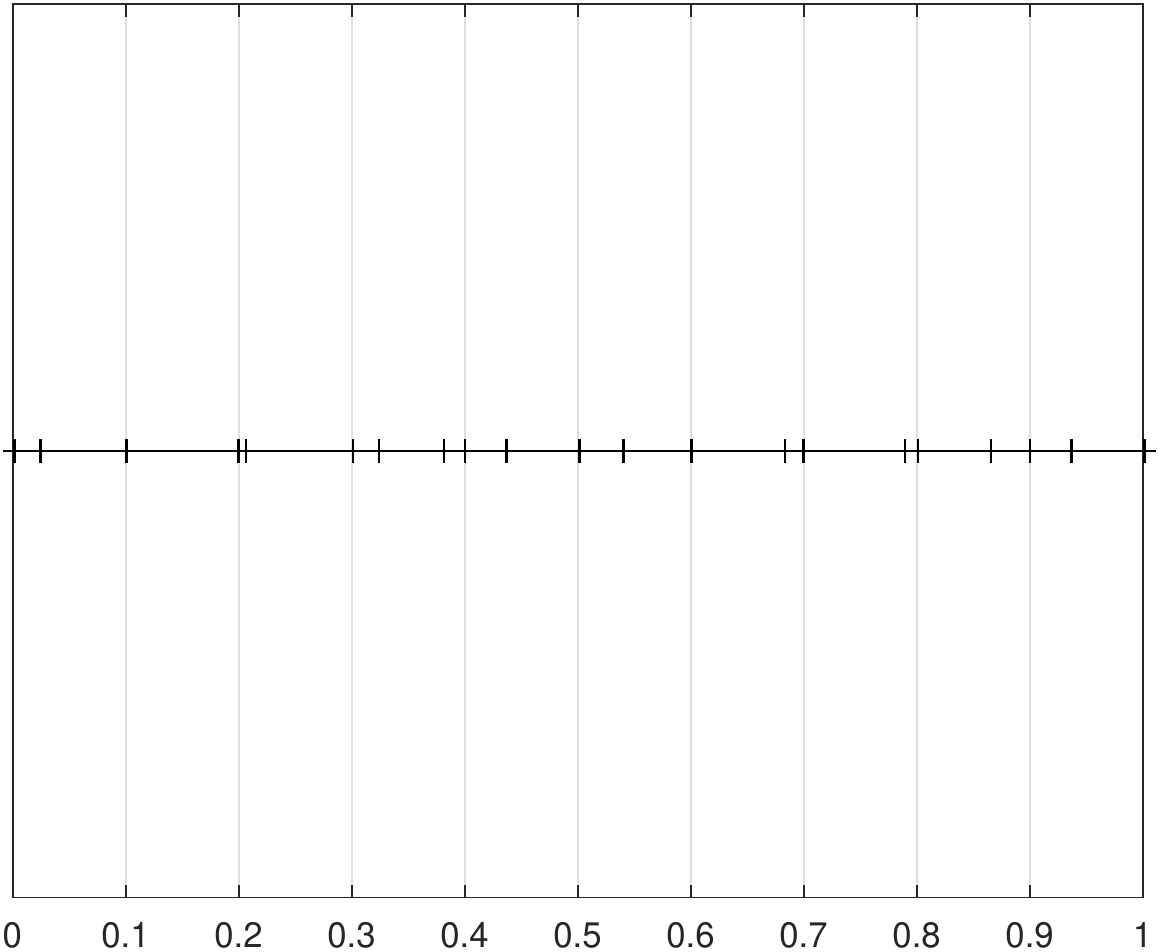}
  }
  \caption{Mesh section at $x=0.5$}
 \label{fig:mesh-section}
 \end{minipage}
\hfill
\end{figure}

In Figs.\,\ref{fig:section-classical-k=1} and
\ref{fig:section-classical-k=2} we plot the section at $x=0.5$ of
the VEM solution for the classical stabilization
\eqref{stab:form:classic} (thick line) together with the exact
solution (thinner line) for $k=1$ and $k=2$ respectively.
\begin{figure}[ht]
\hfill
 \begin{minipage}[b]{0.45\textwidth}
  \centerline{
  \includegraphics[width=\textwidth]{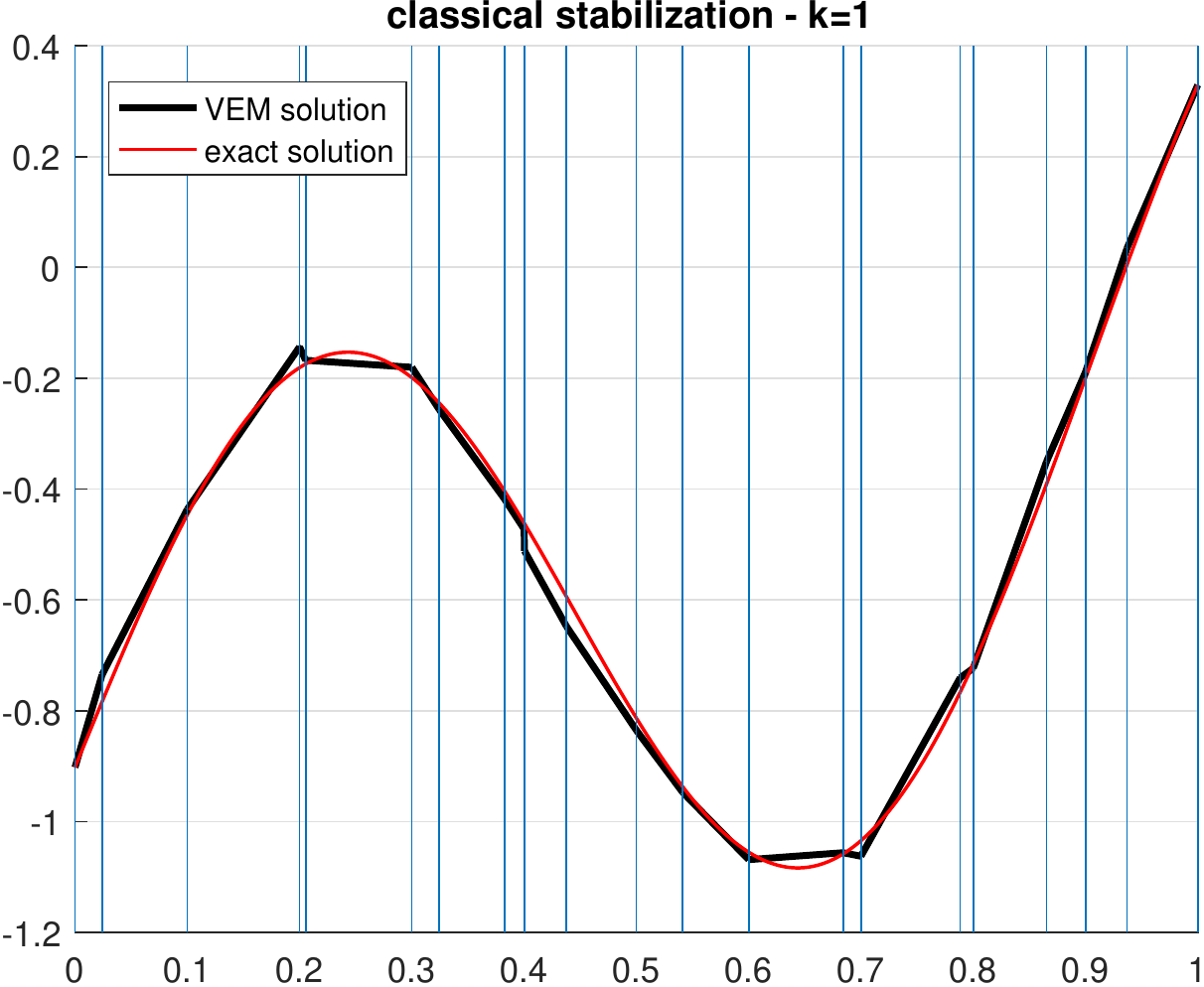}
  }
  \caption{exact solution and VEM solution for $k=1$ and classical stabilization
  \eqref{stab:form:classic}}
 \label{fig:section-classical-k=1}
 \end{minipage}
\hfill
 \begin{minipage}[b]{0.45\textwidth}
  \centerline{
  \includegraphics[width=\textwidth]{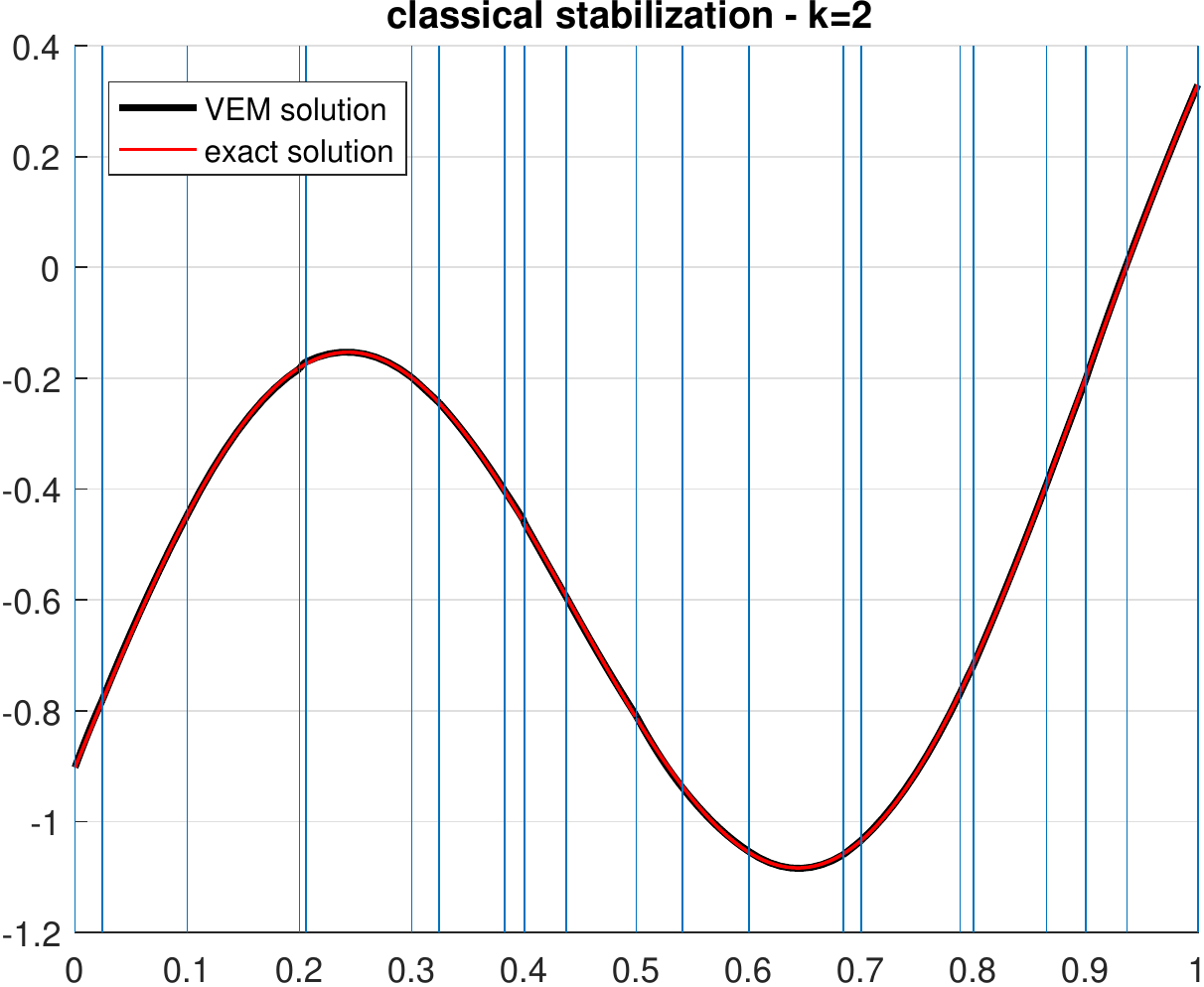}
  }
  \caption{exact solution and VEM solution for $k=2$ and classical stabilization
  \eqref{stab:form:classic}}
 \label{fig:section-classical-k=2}
 \end{minipage}
\hfill
\end{figure}
A careful inspection shows that that in
Fig.\,\ref{fig:section-classical-k=1} there are some small
oscillations in correspondence of the small edges. In Fig.
\ref{fig:section-classical-k=2} the oscillations are no more
visible but are still present.
We reproduce the same experiments in
Figs.\,\ref{fig:section-wriggers-k=1} and
\ref{fig:section-wriggers-k=2} with the boundary stabilization
\eqref{wriggers_ch}.
\begin{figure}[ht]
\hfill
 \begin{minipage}[b]{0.45\textwidth}
  \centerline{
  \includegraphics[width=\textwidth]{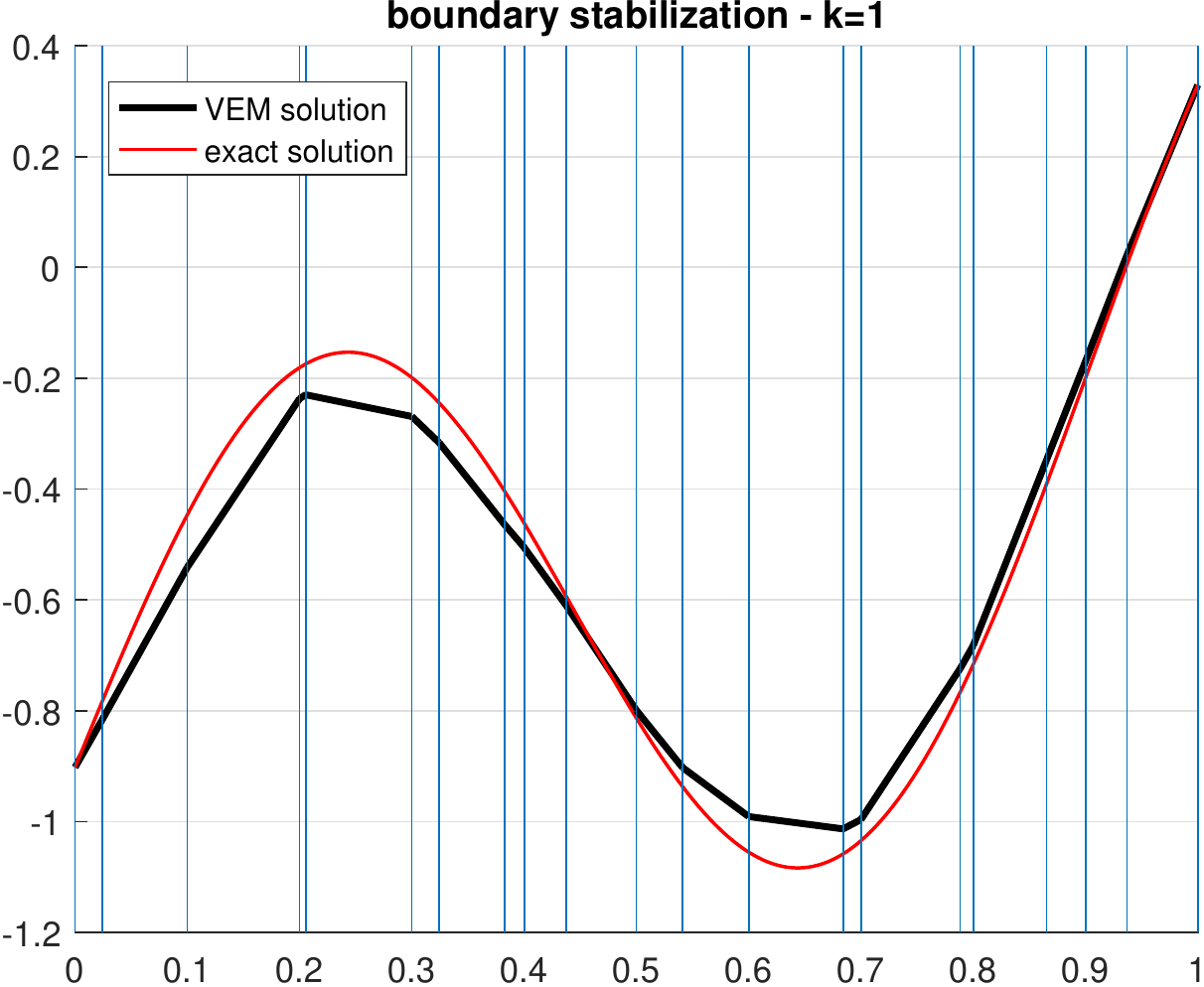}
  }
  \caption{exact solution and VEM solution for $k=1$ and boundary stabilization
  \eqref{wriggers_ch}}
 \label{fig:section-wriggers-k=1}
 \end{minipage}
\hfill
 \begin{minipage}[b]{0.45\textwidth}
  \centerline{
  \includegraphics[width=\textwidth]{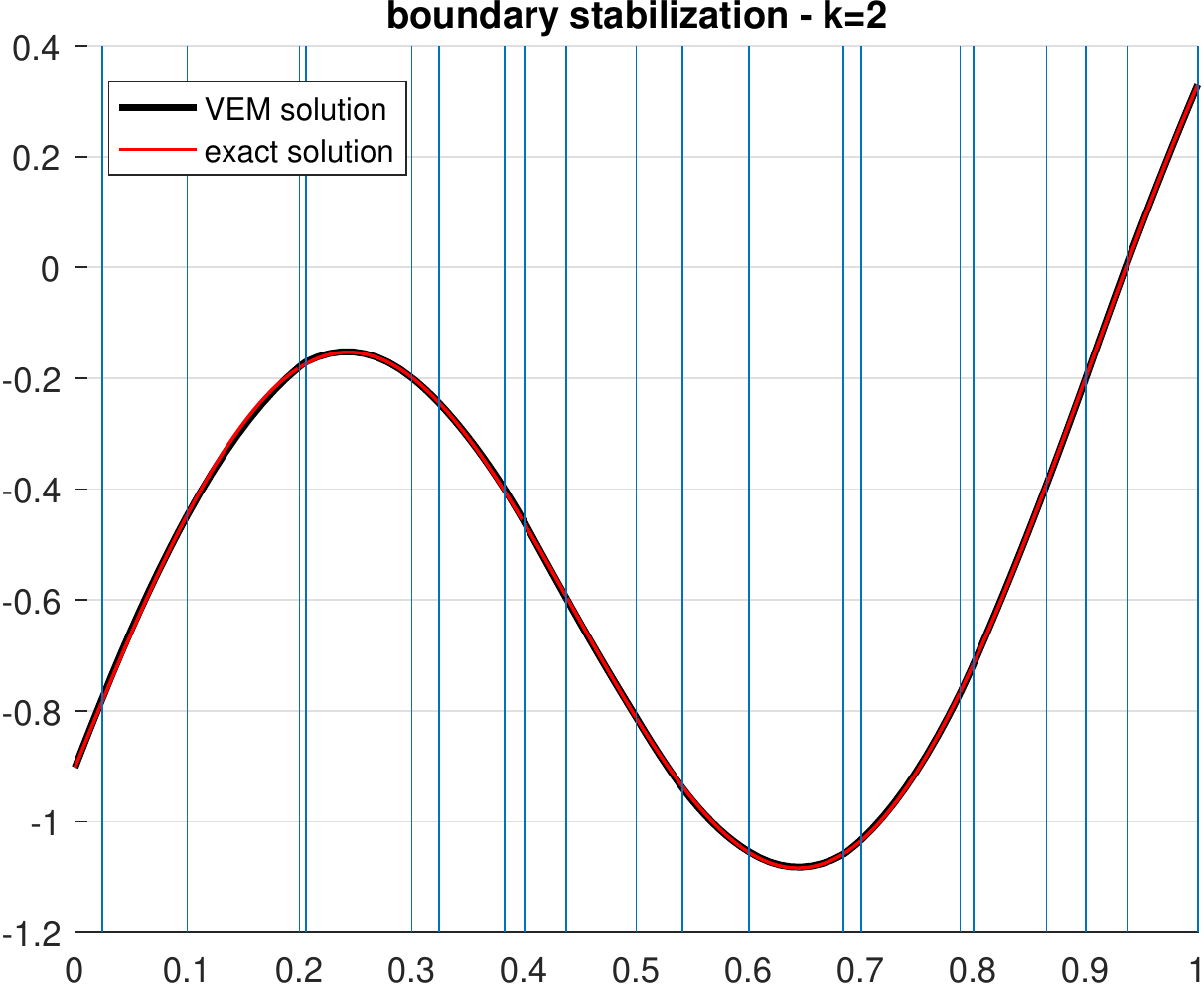}
  }
  \caption{exact solution and VEM solution for $k=2$ and boundary stabilization
  \eqref{wriggers_ch}}
 \label{fig:section-wriggers-k=2}
 \end{minipage}
\hfill
\end{figure}
Now the oscillation have disappeared also for $k=1$ but in this
case the VEM solution seems to be less accurate. The motivation is
that the boundary stabilization \eqref{wriggers_ch} is too strong.
The situation can be improved by taking a smaller stabilization
parameter (see Remark \ref{tau}); in
Figs.\,\ref{fig:section-wriggers-tau=0.1-k=1} and
\ref{fig:section-wriggers-tau=0.1-k=2} we show the same
experiments with $\tau_E=\tau=0.1$. %%
We have developed several further experiments (here not shown)
using different meshes and loading, and choosing $\tau=0.1$ for
$s_E(\cdot,\cdot)$ as in  \eqref{wriggers_ch}: the obtained
results were always accurate. This is a general property of VEM:
the sensitivity of the method with respect to the stabilization
parameter is very mild when considering different meshes and
loading/boundary data. %%
%where now the stabilization term is given by
%$$
%s_E^\partial(v_h,w_h) = \tau \: h_E \int_{\partial E} \partial_s v_h \: \partial_s w_h \: %{\rm d}s
%\qquad \forall v_h, w_h \in V_E, \ \forall E \in \Omega_h .
%$$
%Note that taking $\tau=1$ in the above definition would give back \eqref{wriggers_ch}.
%
\begin{figure}[ht]
\hfill
 \begin{minipage}[b]{0.45\textwidth}
  \centerline{
  \includegraphics[width=\textwidth]{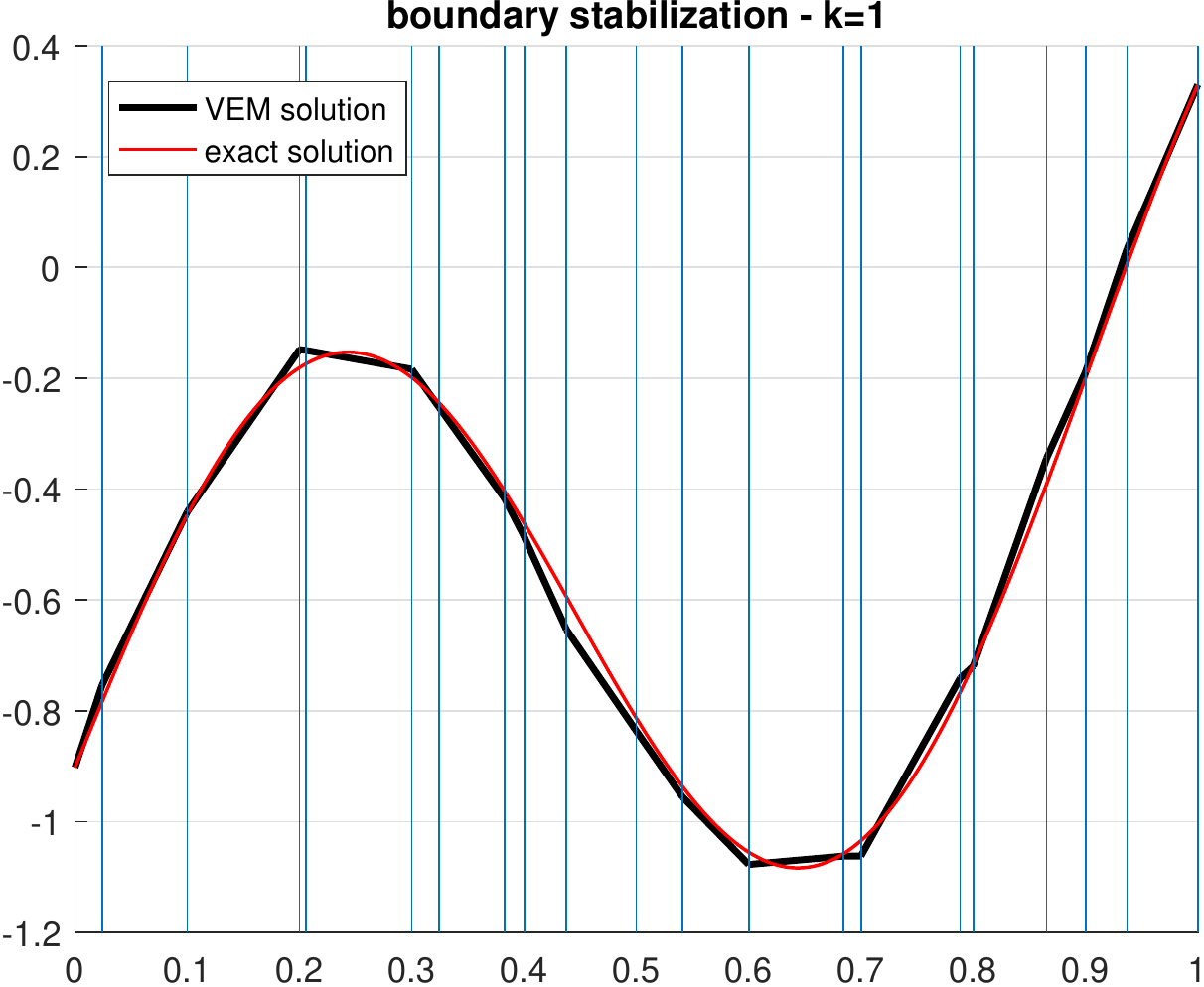}
  }
  \caption{exact solution and VEM solution for $k=1$,
  boundary stabilization \eqref{wriggers_ch} and $\tau=0.1$}
 \label{fig:section-wriggers-tau=0.1-k=1}
 \end{minipage}
\hfill
 \begin{minipage}[b]{0.45\textwidth}
  \centerline{
  \includegraphics[width=\textwidth]{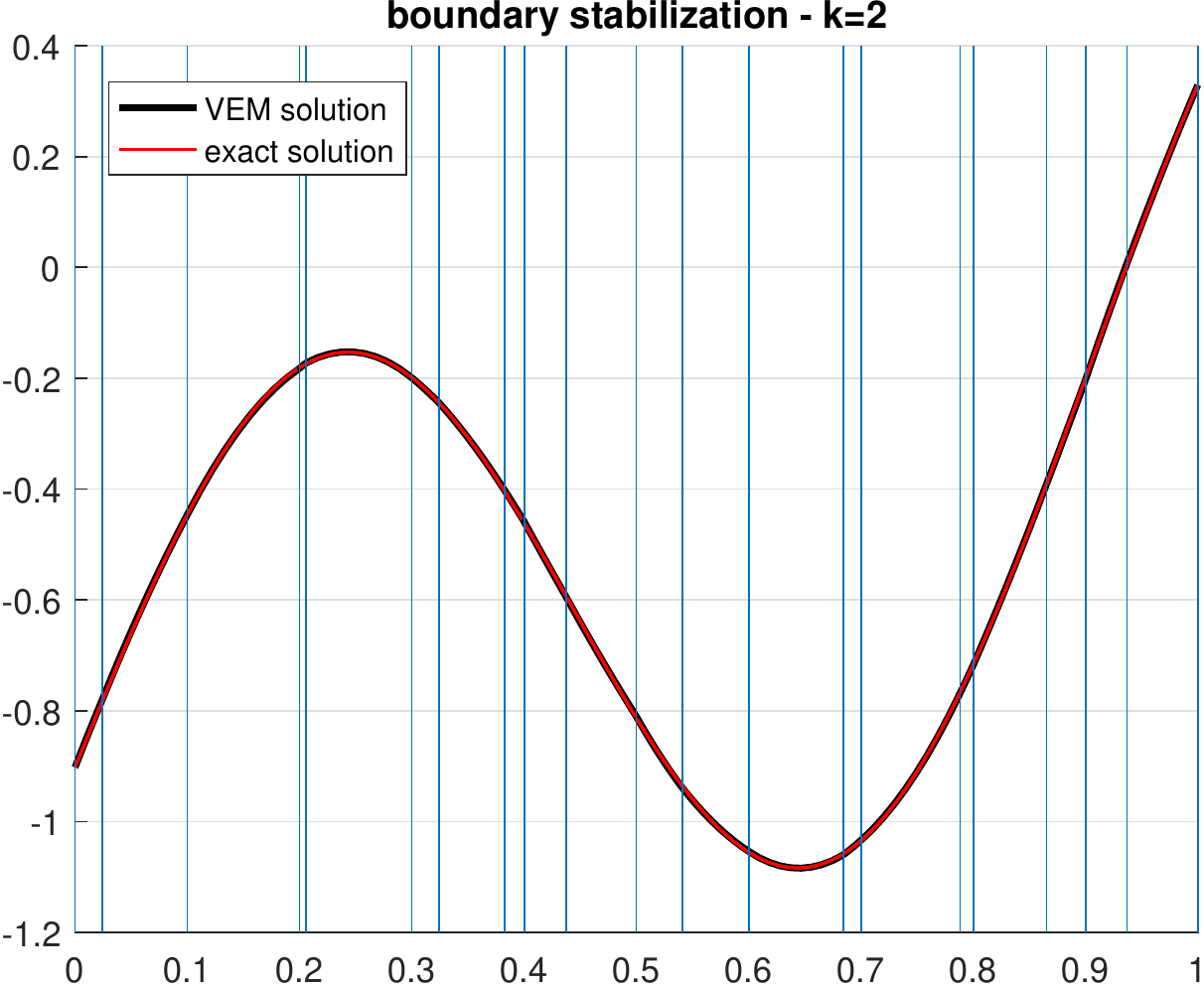}
  }
  \caption{exact solution and VEM solution for $k=2$, boundary stabilization
  \eqref{wriggers_ch} and $\tau=0.1$}
 \label{fig:section-wriggers-tau=0.1-k=2}
 \end{minipage}
\hfill
\end{figure}
%
% We point out that this remedy \textit{does not work} for the classical stabilization
% \eqref{stab:form:classic}, as shown in Figs.\,\ref{fig:section-classical-tau=0.1-k=1} and \ref{fig:section-classical-tau=0.1-k=2}.
%%
Nevertheless, a detailed study on such an issue is beyond the
scope of the present paper. %%

%
%\begin{figure}[ht]
%\hfill
% \begin{minipage}[b]{0.45\textwidth}
%  \centerline{
%  \includegraphics[width=\textwidth]{{section-classical-tau=0.1-k=1}.pdf}
%  }
%  \caption{VEM solution and exact solution for $k=1$, classical stabilization and $\tau=0.1$}
% \label{fig:section-classical-tau=0.1-k=1}
% \end{minipage}
%\hfill
% \begin{minipage}[b]{0.45\textwidth}
%  \centerline{
%  \includegraphics[width=\textwidth]{{section-classical-tau=0.1-k=2}.pdf}
%  }
%  \caption{VEM solution and exact solution for $k=2$, classical stabilization and $\tau=0.1$}
% \label{fig:section-classical-tau=0.1-k=2}
% \end{minipage}
%\hfill
%\end{figure}
%

% ---------------------------------------------------------
\subsection{Convergence in $H^1$}

We will show, in a loglog scale, the convergence curves of the
error in the $H^1$ seminorm between the exact solution $\uex$ and
the solution $u_h$ given by the Virtual Element Method. As the VEM
solution $u_h$ is not explicitly known inside the elements, we
compare $\nabla\uex$ with the elementwise $L^2-$projection of
$\nabla u_h$ onto $\Pp_{k-1}$, that is, with $\Pi^0_{k-1}\nabla
u_h$. It is easy to see that this latter quantity can indeed be
computed starting from the degrees of freedom of $u_h$.
For the convergence test we consider four sequences of meshes.% described below.

The first sequence of meshes (labelled \texttt{square}) is simply
a decomposition of the domain in $4\times4$, $8\times8$,
$16\times16$ and $32\times32$ equal squares, and the second one
(labelled \texttt{hexagon}) is a decomposition of the domain in
$8\times10$, $18\times20$, $26\times30$, $34\times40$ and
$44\times50$ (almost) regular hexagons.
The first meshes of the two sequences are shown in Fig.
\ref{fig:square} and in Fig. \ref{fig:hexagon} respectively.
\begin{figure}[ht]
\hfill
 \begin{minipage}[b]{0.45\textwidth}
  \centerline{
  \includegraphics[width=\textwidth]{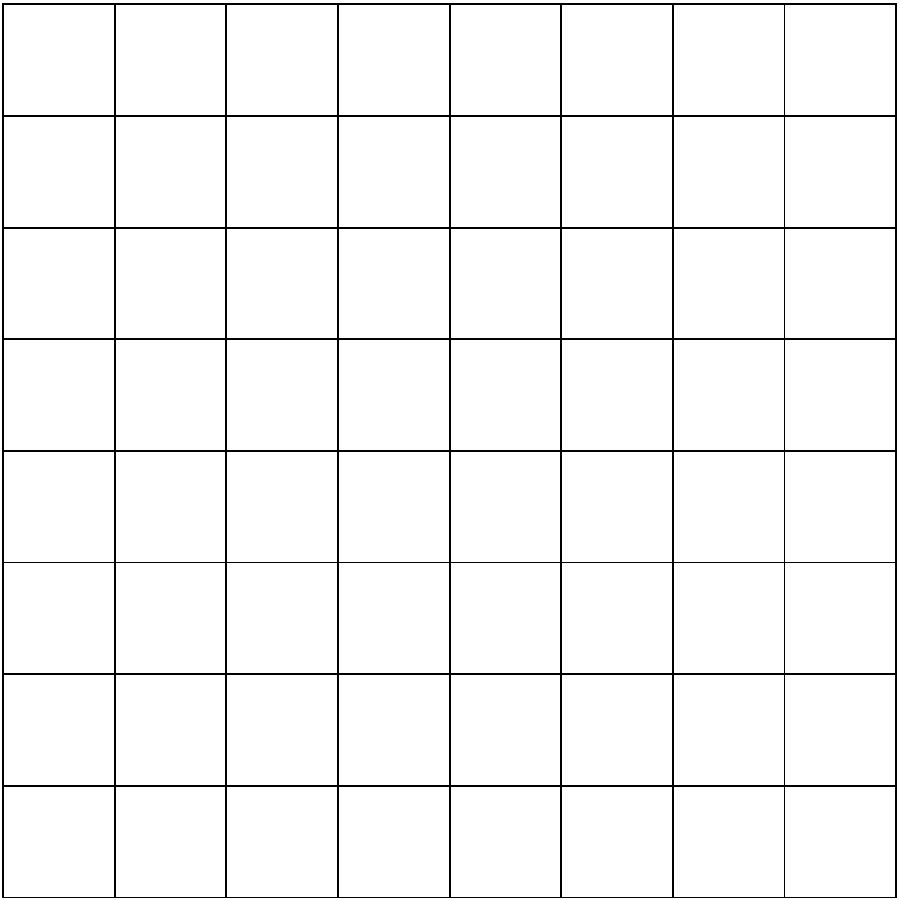}
  }
  \caption{\texttt{square} mesh}
 \label{fig:square}
 \end{minipage}
\hfill
 \begin{minipage}[b]{0.45\textwidth}
  \centerline{
  \includegraphics[width=\textwidth]{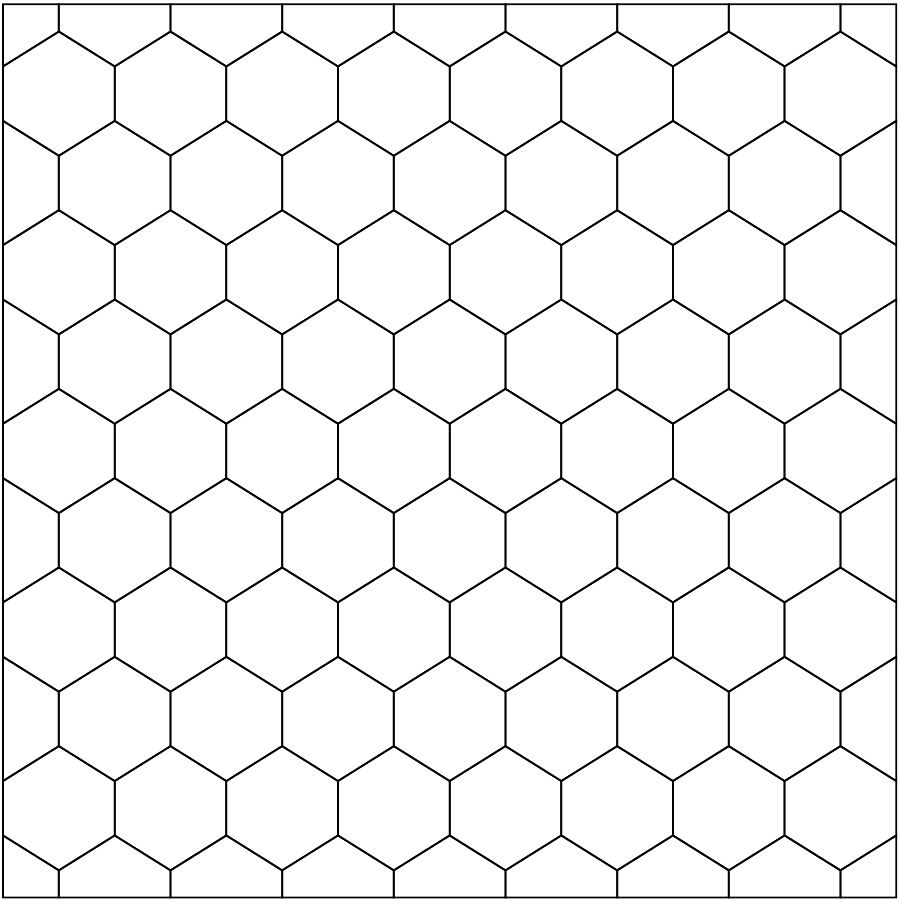}
  }
  \caption{\texttt{hexagon} mesh}
 \label{fig:hexagon}
 \end{minipage}
\hfill
\end{figure}

The third sequence of meshes (labelled \texttt{Lloyd-0}) is a
random Voronoi polygonal tessellation of the unit square in 25,
100, 400 and 1600 polygons.
The fourth sequence (labelled \texttt{Lloyd-100}) is obtained
starting from the previous one and performing 100 Lloyd iterations
leading to a Centroidal Voronoi Tessellation (CVT) (see e.g.
\cite{Du:Faber99}).
The 100-polygon mesh of each family is shown in
Fig.~\ref{fig:Lloyd-0} (\texttt{Lloyd-0}) and in
Fig.~\ref{fig:Lloyd-100} (\texttt{Lloyd-100}) respectively.
\begin{figure}[ht]
\hfill
 \begin{minipage}[b]{0.45\textwidth}
  \centerline{
  \includegraphics[width=\textwidth]{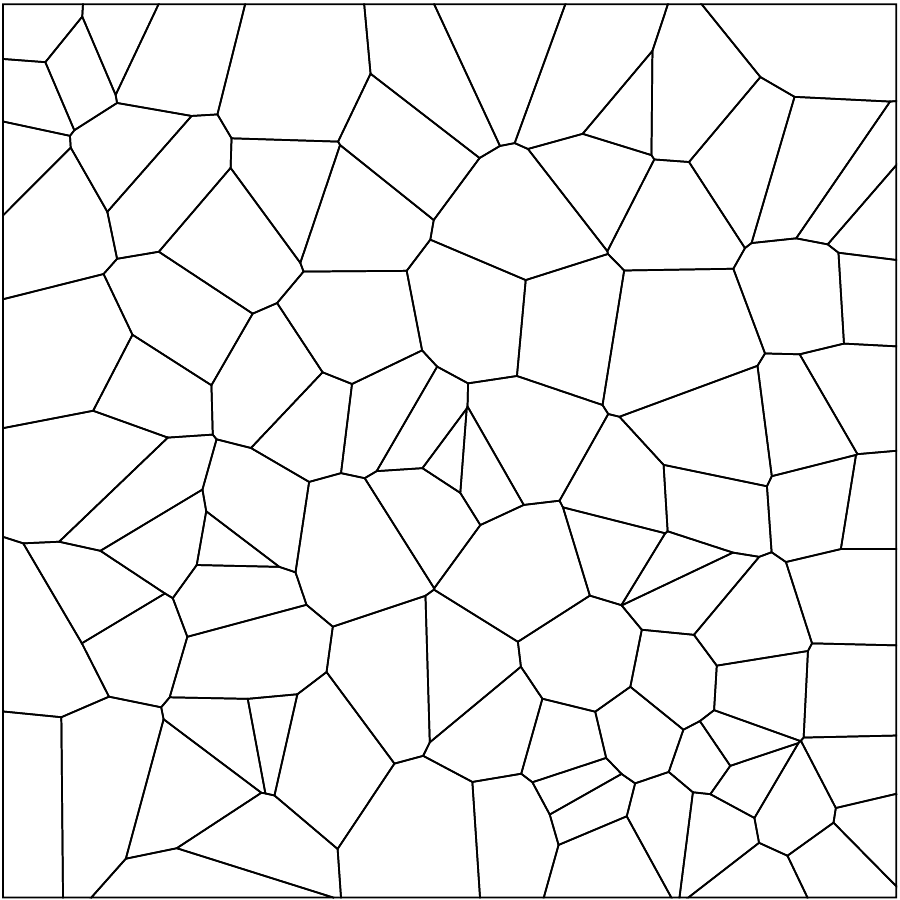}
  }
  \caption{\texttt{Lloyd-0} mesh}
 \label{fig:Lloyd-0}
 \end{minipage}
\hfill
 \begin{minipage}[b]{0.45\textwidth}
  \centerline{
  \includegraphics[width=\textwidth]{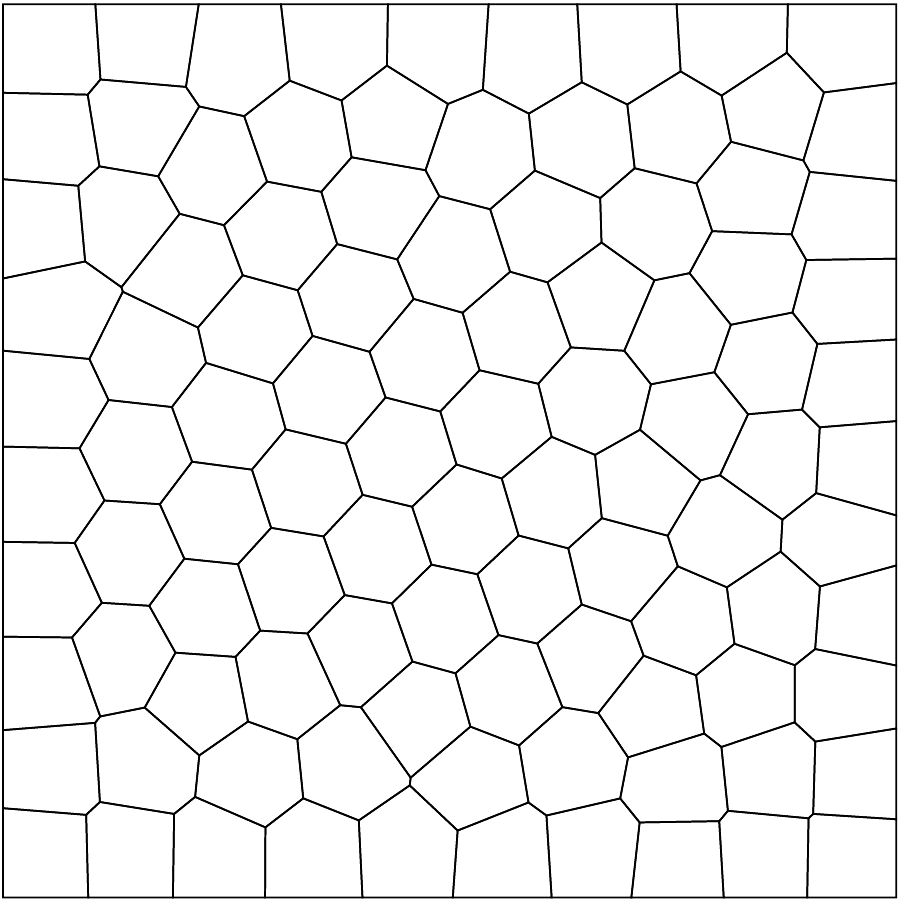}
  }
  \caption{\texttt{Lloyd-100} mesh}
 \label{fig:Lloyd-100}
 \end{minipage}
\hfill
\end{figure}

In the figures from \ref{fig:k=1-square-H1} to
\ref{fig:k=1-LLoyd-100-H1} we plot for $k=1$ (low order case) the
$H^1$ error on each mesh family as a function of the mean diameter
$h$ of the polygons. We consider the classical stabilization
\eqref{stab:form:classic} (solid line), the boundary stabilization
\eqref{wriggers_ch} with $\tau=1$ (dotted line), and the boundary
stabilization \eqref{wriggers_ch} with $\tau=0.1$ (dashed line).
In the figures from \ref{fig:k=5-square-H1} to
\ref{fig:k=5-LLoyd-100-H1} we finally plot the same values for
$k=5$ (as a sample high order case).

We observe that as $h$ goes to zero all stabilizations behave very
similarly, namely as $O(h^k)$, as predicted by the theory.

%
% k=1
%
\begin{figure}[ht]
\hfill
 \begin{minipage}[b]{0.49\textwidth}
  \centerline{
  \includegraphics[width=\textwidth]{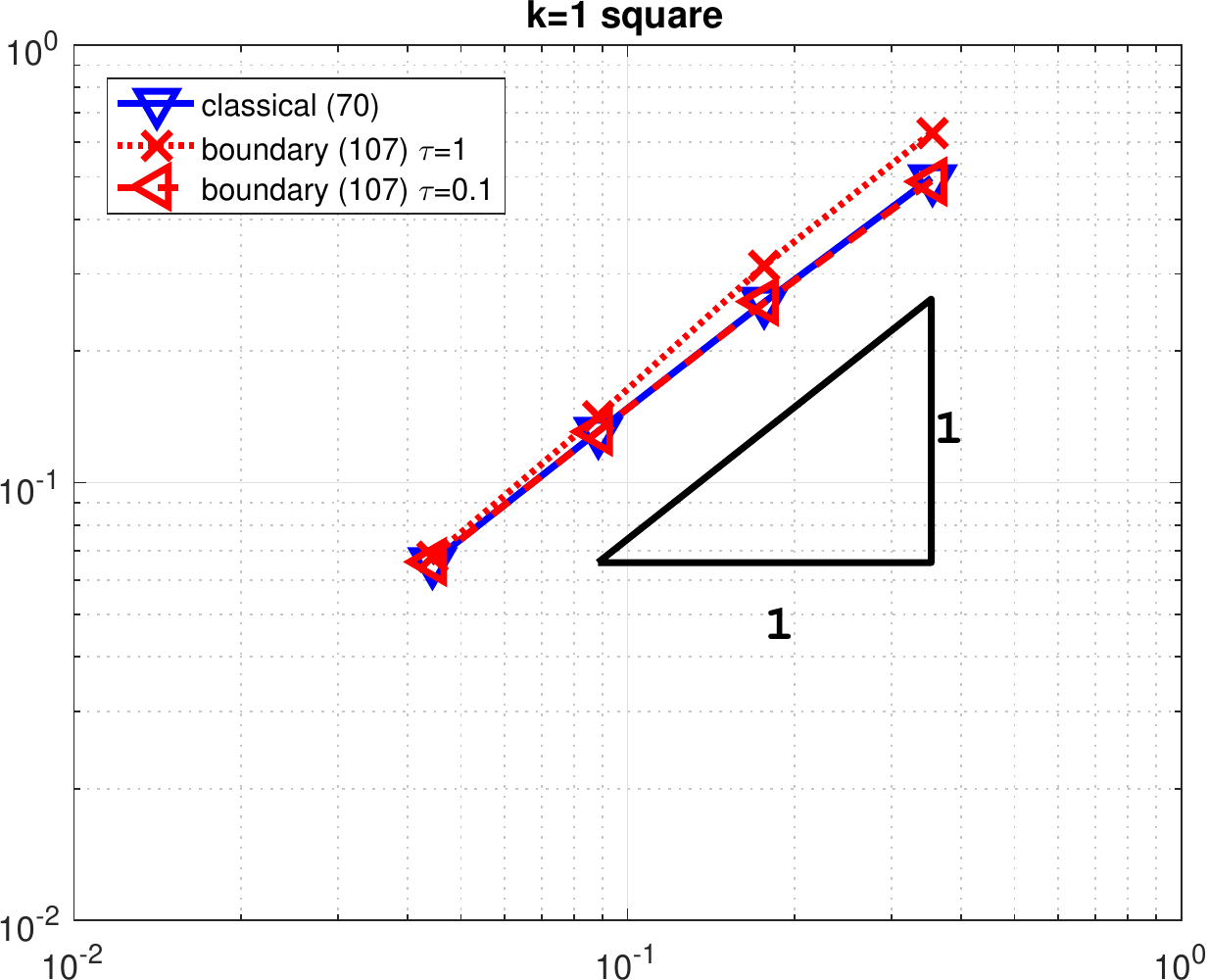}
  }
  \caption{$k=1$, \texttt{square} mesh}
 \label{fig:k=1-square-H1}
 \end{minipage}
\hfill
 \begin{minipage}[b]{0.49\textwidth}
  \centerline{
  \includegraphics[width=\textwidth]{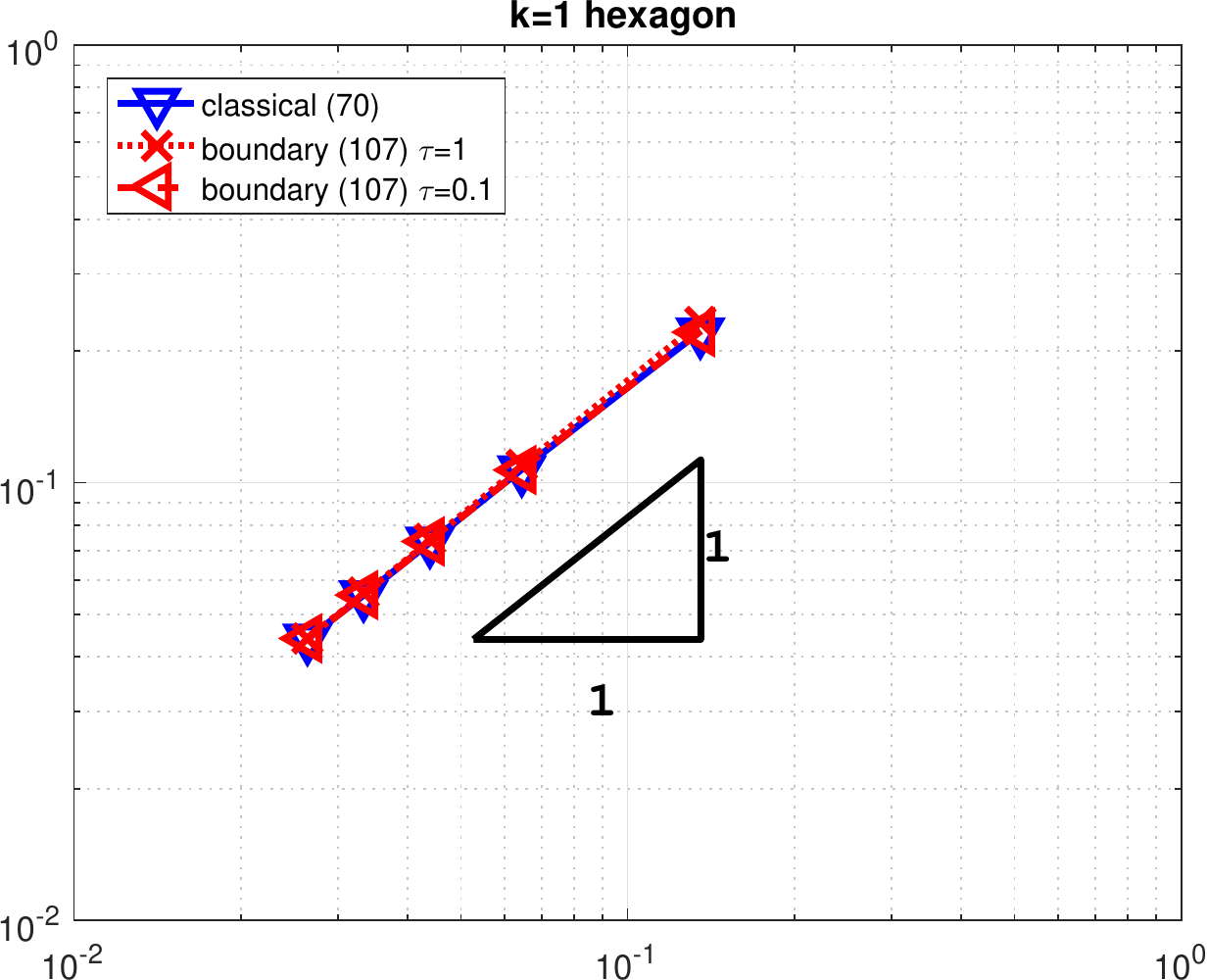}
  }
  \caption{$k=1$, \texttt{hexagon} mesh}
 \label{fig:k=1-hexagon-H1}
 \end{minipage}
\hfill
\end{figure}
\begin{figure}[ht]
\hfill
 \begin{minipage}[b]{0.49\textwidth}
  \centerline{
  \includegraphics[width=\textwidth]{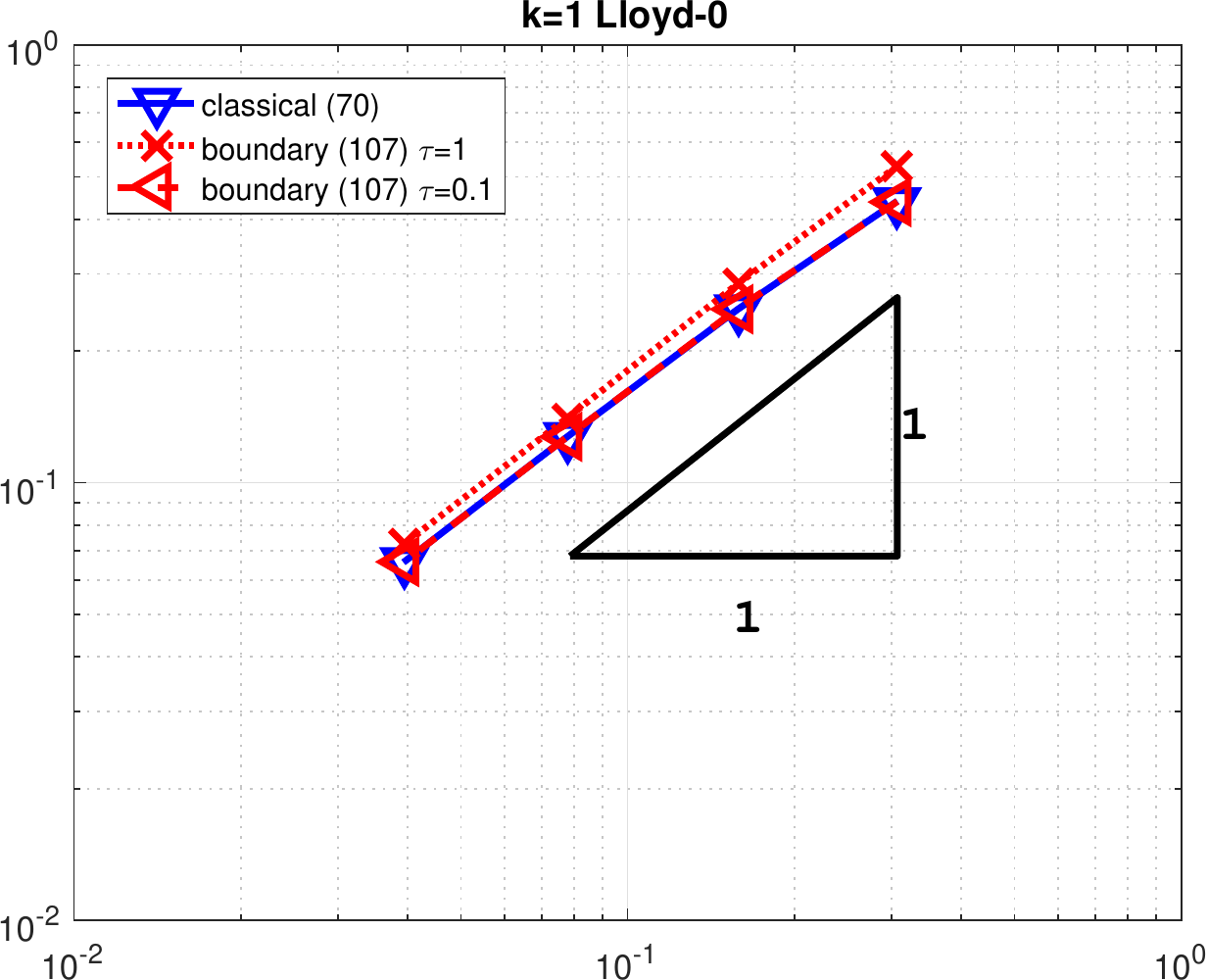}
  }
  \caption{$k=1$, \texttt{Lloyd-0} mesh}
 \label{fig:k=1-LLoyd-0-H1}
 \end{minipage}
\hfill
 \begin{minipage}[b]{0.49\textwidth}
  \centerline{
  \includegraphics[width=\textwidth]{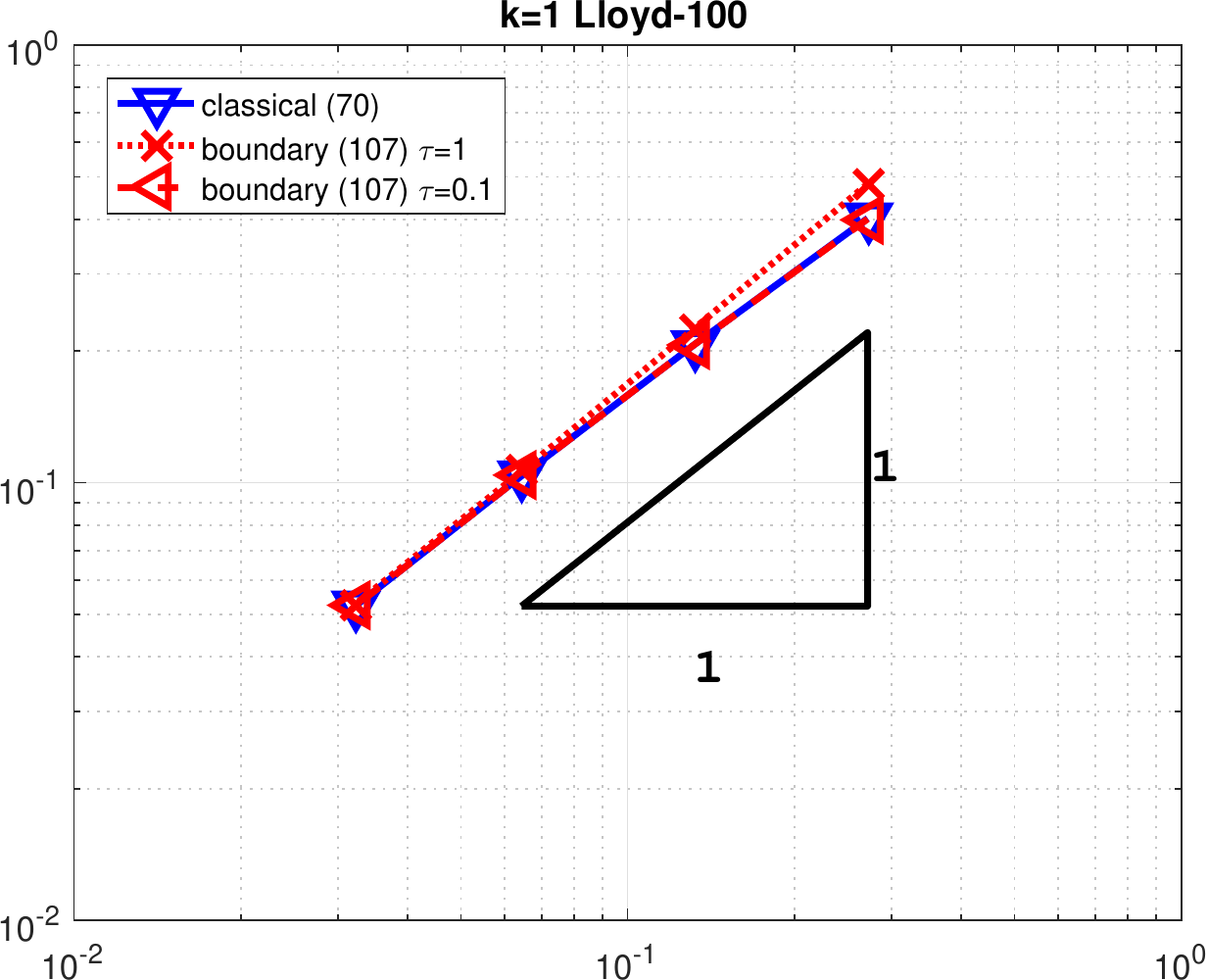}
  }
  \caption{$k=1$, \texttt{Lloyd-100} mesh}
 \label{fig:k=1-LLoyd-100-H1}
 \end{minipage}
\hfill
\end{figure}
\begin{figure}[ht]
\hfill
 \begin{minipage}[b]{0.49\textwidth}
  \centerline{
  \includegraphics[width=\textwidth]{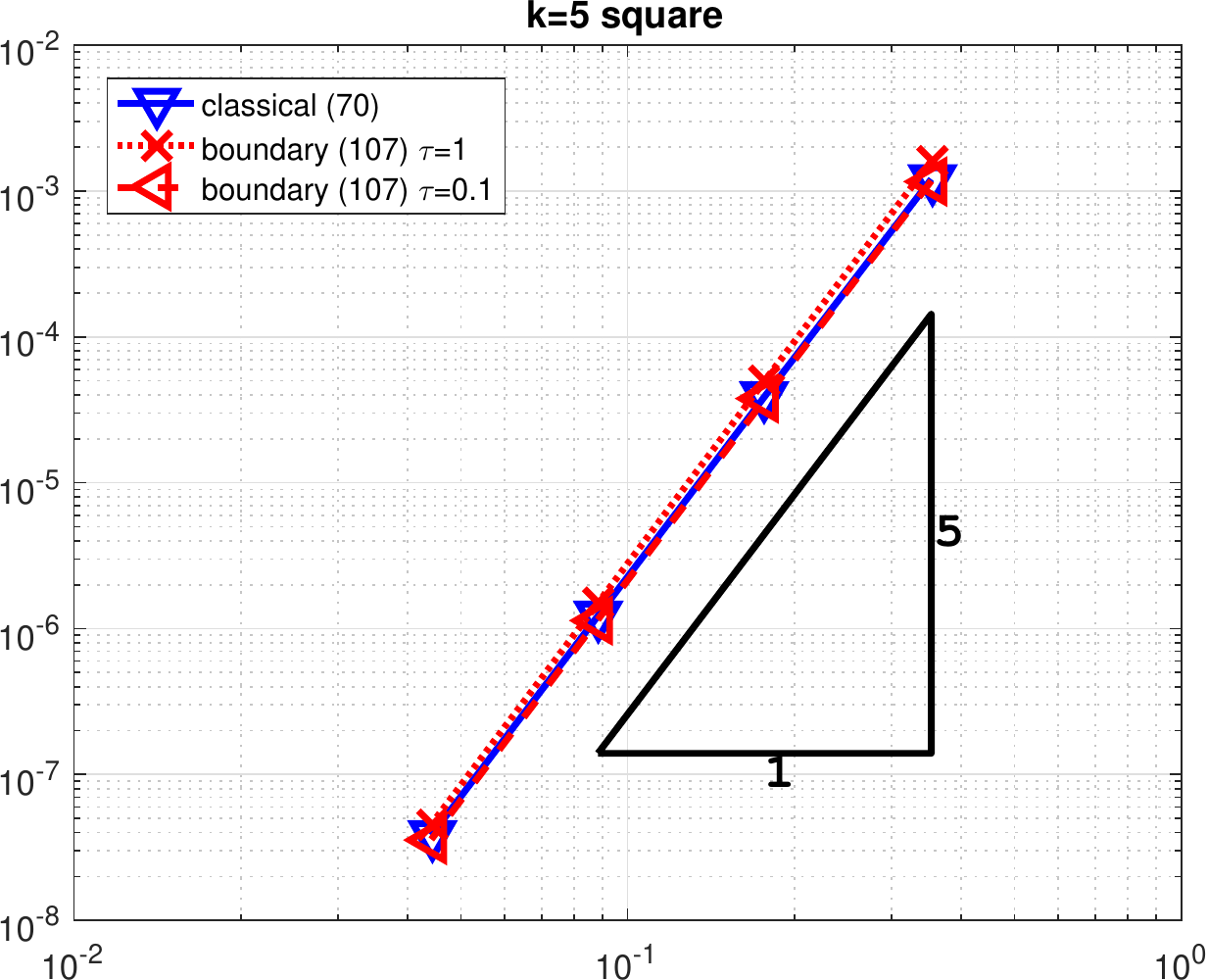}
  }
  \caption{$k=5$, \texttt{square} mesh}
 \label{fig:k=5-square-H1}
 \end{minipage}
\hfill
 \begin{minipage}[b]{0.49\textwidth}
  \centerline{
  \includegraphics[width=\textwidth]{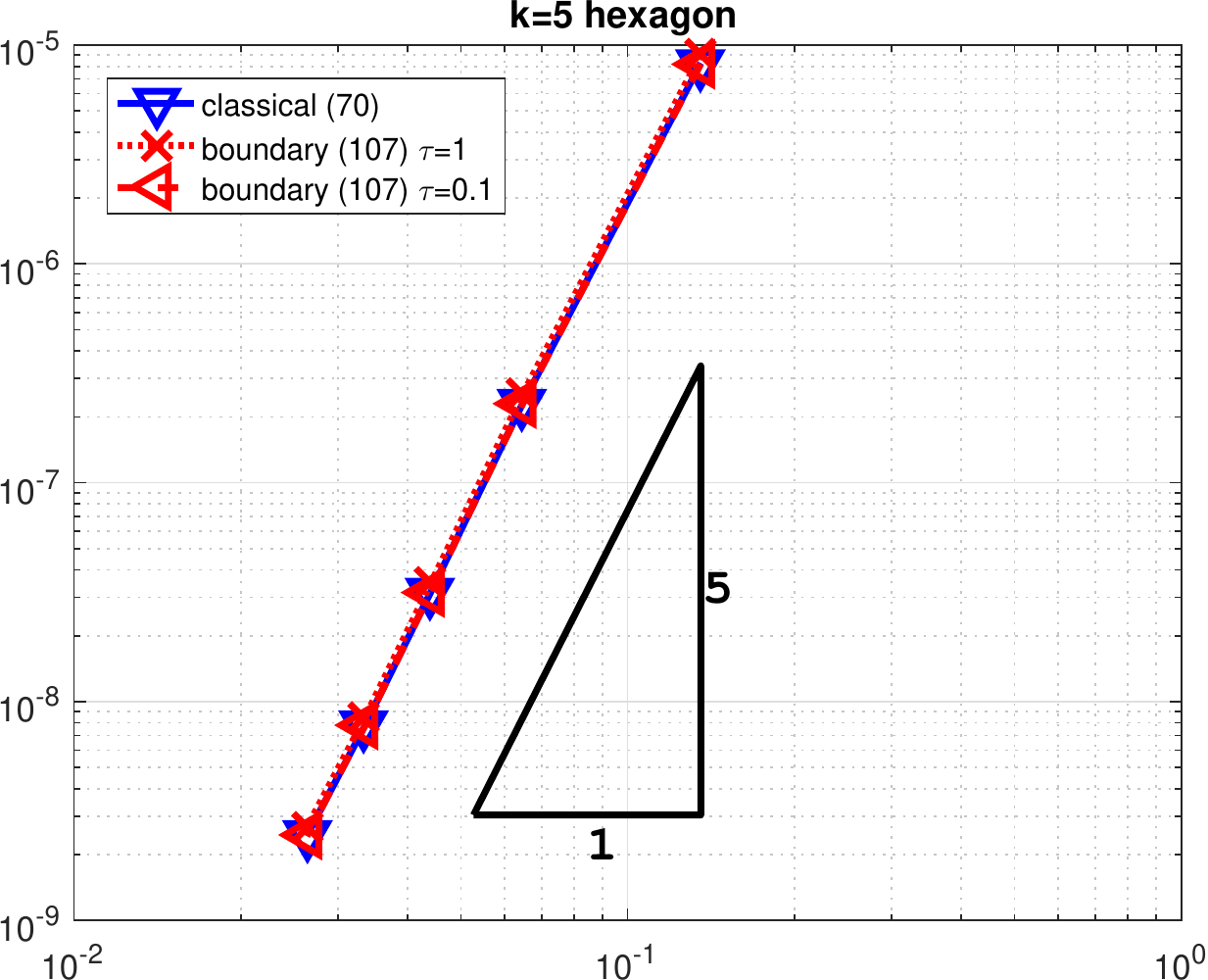}
  }
  \caption{$k=5$, \texttt{hexagon} mesh}
 \label{fig:k=5-hexagon-H1}
 \end{minipage}
\hfill
\end{figure}
\begin{figure}[ht]
\hfill
 \begin{minipage}[b]{0.49\textwidth}
  \centerline{
  \includegraphics[width=\textwidth]{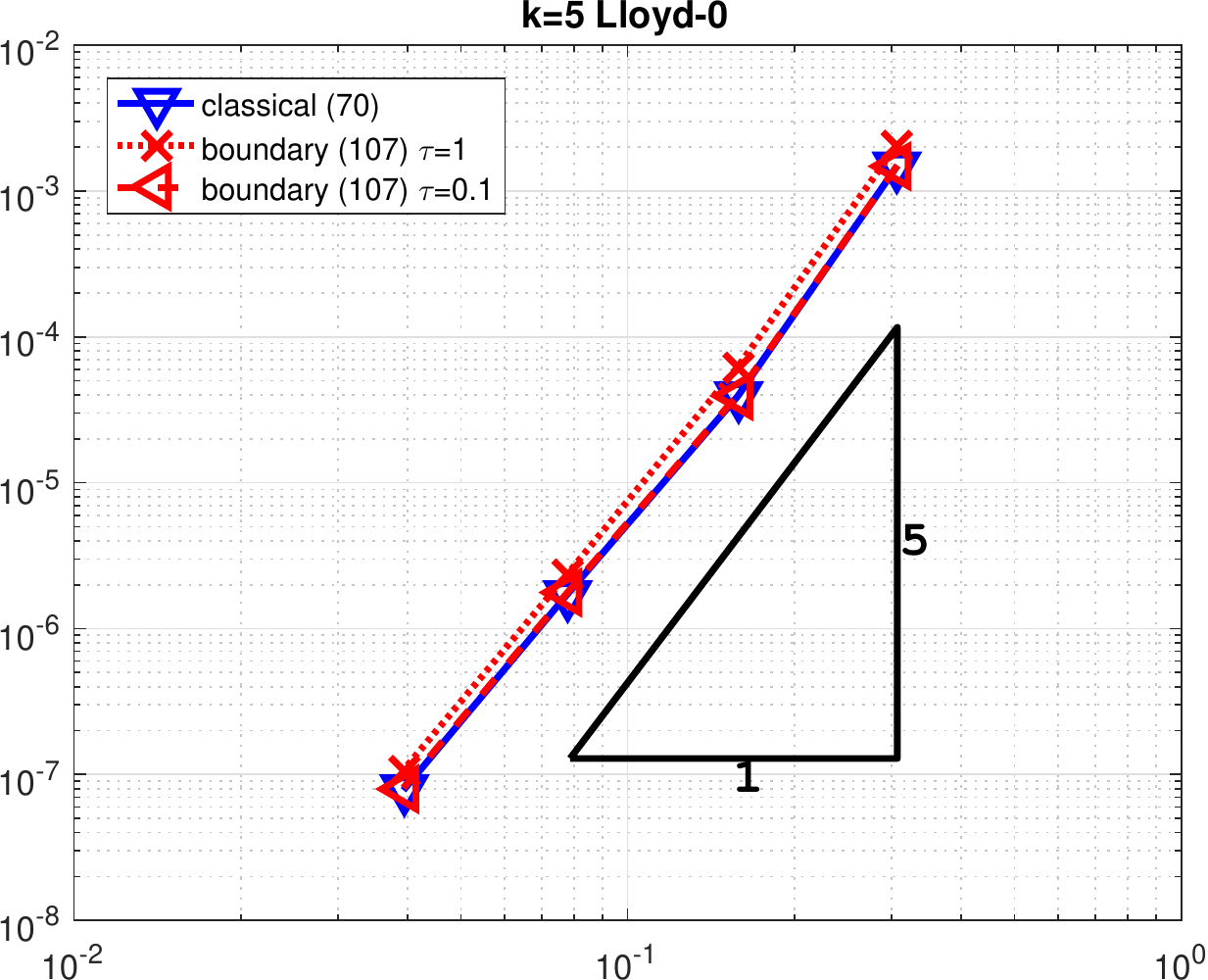}
  }
  \caption{$k=5$, \texttt{Lloyd-0} mesh}
 \label{fig:k=5-LLoyd-0-H1}
 \end{minipage}
\hfill
 \begin{minipage}[b]{0.49\textwidth}
  \centerline{
  \includegraphics[width=\textwidth]{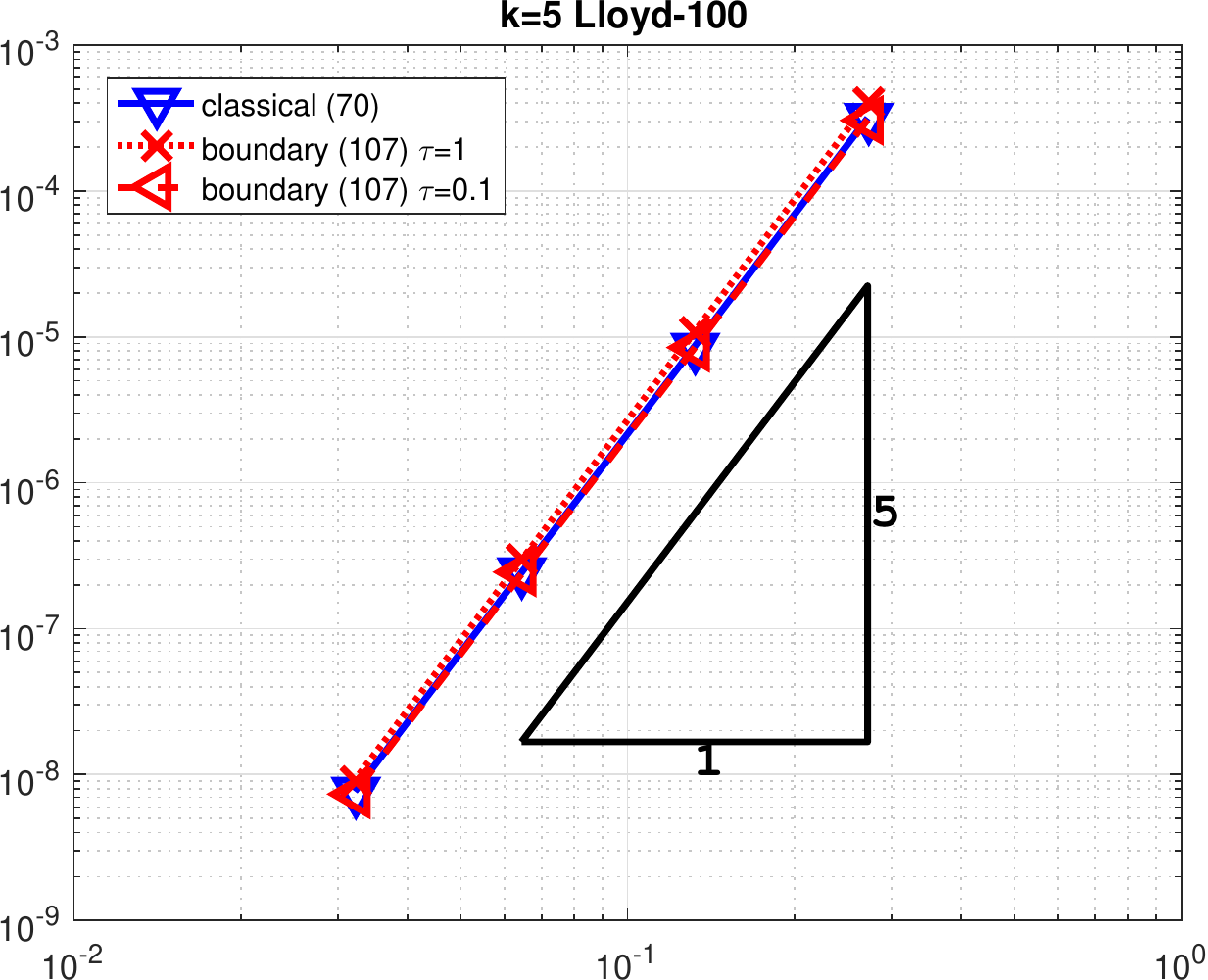}
  }
  \caption{$k=5$, \texttt{Lloyd-100} mesh}
 \label{fig:k=5-LLoyd-100-H1}
 \end{minipage}
\hfill
\end{figure}

\begin{center}
{\large {\bf Aknowledgements}}
\end{center}
The first author was partially supported by the European Research Council through the H2020 Consolidator Grant (grant no. 681162) CAVE -- Challenges and Advancements in Virtual Elements. This support is gratefully acknowledged. \\
The third authors was partially supported by the research funds of
the University of Milano-Bicocca. This support is gratefully acknowledged. \\
All the authors were partially supported by IMATI-CNR. This support is gratefully acknowledged.

\medskip

%%%%%%%%%%%%%%%%%%%%%%%%%%%%%

\bibliographystyle{amsplain}

\bibliography{general-bibliography}

%%%%%%%%%%%%%%%%%%%%%%%%%%%%%%%%%%%%%%%%%%%%%%%%
 \end{document}